\documentclass[graybox]{svmult}

\usepackage{mathptmx}       % selects Times Roman as basic font
\usepackage{helvet}         % selects Helvetica as sans-serif font
\usepackage{courier}        % selects Courier as typewriter font
\usepackage{type1cm}        % activate if the above 3 fonts are
\usepackage{makeidx}         % allows index generation
\usepackage{graphicx}        % standard LaTeX graphics tool
\usepackage{multicol}        % used for the two-column index
\usepackage[bottom]{footmisc}% places footnotes at page bottom

\usepackage{amsmath,amssymb,color}
\makeindex             % used for the subject index
\newcommand{\ensemblenombre}[1]{\mathbb{#1}}

\newcommand{\Z}{\ensemblenombre{Z}}

\newcommand{\R}{\ensemblenombre{R}}
\newcommand{\C}{\mathcal{C}}

\newcommand{\A}{\mathcal{A}}

\newcommand{\DD}{\mathcal{D}}

\newcommand{\F}{{\mathcal{F}}}
\renewcommand{\P}{{\mathcal{P}}}
\renewcommand{\I}{{I}}
\newcommand{\FC}{{\mathcal{F}_\C}}

\newcommand{\1}{{\bf 1}}
\renewcommand{\L}{\Lambda}
\renewcommand{\D}{\Delta}

\renewcommand{\l}{\lambda}
\newcommand{\Rd}{\ensemblenombre{R}^d}
\newcommand{\g}{\gamma}
\newcommand{\G}{\Gamma}
\newcommand{\tL}{{\tau_\mathcal{L}}}
\newcommand{\z}{\zeta}

\begin{document}

\title*{Introduction to the theory of Gibbs point processes}
\author{DEREUDRE David}
\institute{DEREUDRE David, University Lille 1, \email{david.dereudre@univ-lille1.fr}}

\maketitle

\abstract*{rrrrrrrrrrrrrrrrrrrrEach chapter should be preceded by an abstract (10--15 lines long) that summarizes the content. The abstract will appear \textit{online} at \url{www.SpringerLink.com} and be available with unrestricted access. This allows unregistered users to read the abstract as a teaser for the complete chapter. As a general rule the abstracts will not appear in the printed version of your book unless it is the style of your particular book or that of the series to which your book belongs.
Please use the 'starred' version of the new Springer \texttt{abstract} command for typesetting the text of the online abstracts (cf. source file of this chapter template \texttt{abstract}) and include them with the source files of your manuscript. Use the plain \texttt{abstract} command if the abstract is also to appear in the printed version of the book.}

\abstract{
The Gibbs point processes (GPP) constitute  a large class of point processes with interaction between the points. The interaction can be attractive, repulsive, depending on geometrical features whereas the null interaction is associated  with the so-called Poisson point process. In a first part of this mini-course, we present several aspects of finite volume GPP defined on a bounded window in $\R^d$. In a second part, we introduce the more complicated formalism of infinite volume GPP defined on the full space $\R^d$. Existence, uniqueness and non-uniqueness  of GPP are non-trivial questions which we treat here  with completely self-contained proofs. The DLR equations, the GNZ equations and the variational principle are presented as well. Finally we investigate the estimation of parameters. The main standard estimators (MLE, MPLE, Takacs-Fiksel  and variational estimators) are presented and we prove their consistency. For sake of simplicity, during all the mini-course, we consider only the case of finite range interaction and  the setting of marked points is not presented. 
}

\section*{Introduction}

The spatial point processes are well studied objects in probability theory and statistics for modelling and analysing spatial data which appear in several disciplines as statistical mechanics, material science, astronomy, epidemiology, plant ecology, seismology, telecommunication, and others \cite{Baddeleyetal,DVJ}. There exist many models of such random points configurations in space and the most popular one is surely the Poisson point process. It corresponds to the natural way of producing independent locations of points in space without interaction. For dependent random structures, we can mention for instance the Cox processes, determinantal point processes, Gibbs point processes, etc. None of them is established as the most relevant model for applications. In fact the choice of the model depends on the nature of the dataset, the knowledge of (physical or biological) mechanisms producing the pattern, the aim of the study (theoretical, applied or numerical).

In this mini-course, we focus on Gibbs point processes (GPP) which constitute a large class of points processes, able to fit several kinds of patterns and which provide a clear interpretation of the interaction between the points, such as attraction or repulsion depending on their relative position. Note that this class is particularly large since several point processes can be represented as GPP (see \cite{GeorgiiYoo,Kozlov} for instance). The main disadvantage of GPP is the complexity of the model due to an intractable normalizing constant which appears in the local conditional densities. Therefore their analytical studies are in general based on implicit equilibrium equations which lead to complicated  and delicate analysis. Moreover, the theoretical results which are needed to investigate the Gibbs point process theory are scattered  across several publications or books. The aim of this mini-course is to provide a solid and self-contained theoretical basis for understanding deeply the Gibbs point process theory. The results are in general not exhaustive but the main ideas and tools are presented in accordance with modern and recent developments. The main strong restriction here involves the range of the interaction, which is assumed to be finite. The infinite range interaction requires the introduction of tempered configuration spaces and for sake of simplicity we decided to avoid this level of complexity. The mini-course is addressed for Master and Phd students and also for researchers who want to discover or investigate the domain. The manuscript is based on a mini-course given during the conference of GDR 3477 g\'eom\'etrie stochastique, at university of Nantes in April 2016.

In a first section, we introduce the finite volume GPP  on a bounded window $\L\subset \Rd$. They are simply defined as point processes in $\L$ whose the distributions are absolutely continuous with respect to the Poisson point process distribution. The unnormalized densities are of form $z^Ne^{-\beta H}$, where $z$ and $\beta$ are positive parameters (called respectively activity and inverse temperature), $N$ is the number of points and $H$ an energy function. Clearly, these distributions favour (or penalize) configurations with low (or high) energy $E$. This distortion strengthens as $\beta$ is large. The parameter $z$ allows to tune the mean number of points. This setting is relatively simple since all the objects are defined explicitly. However, the intractable normalization constant is ever a problem and most of quantities are not computable. Several standard notions (DLR and GNZ equations, Ruelle's estimates, etc.) are treated in this first section as a preparation for the more complicated setting of infinite volume GPP developed in the second section. Note that we do not present the setting of  marked Gibbs point processes in order to keep  the notations as simple as possible. However, all the results can be easily extended in this case. 

In a second section, we present the theory of infinite volume GPP in  $\Rd$. There are several motivations for studying such infinite volume regime. Firstly, the GPP are the standard models in statistical physics for modelling systems with a large number of interacting particles (around $10^{23}$ according to the Avogadro's number). Therefore, the case where the number of particles is infinite is an idealization of this setting and furnishes microscopic descriptions of  gas, liquid or solid. Macroscopic quantities like the density of particles, the pressure and the mean energy are consequently easily defined by mean values or laws of large numbers. Secondly, in the spatial statistic context, the asymptotic properties of estimators or tests are obtained when the observation window tends to the full space $\Rd$. This strategy requires the existence of infinite volume models. Finally, since the infinite volume GPP are stationary (shift invariant) in $\Rd$, several powerful tools, as the ergodic theorem or the central limit Theorem for mixing field, are available in this infinite volume regime.  

The infinite volume Gibbs measures are defined by a collection of implicit DLR equations (Dobrushin, Lanford and Ruelle). The existence, uniqueness and non-uniqueness are non trivial questions which we treat in depth with self-contained proofs in this second section. The phase transition between uniqueness and non uniqueness is one of the most difficult conjectures in statistical physics. This phenomenon is expected to occur for all standard interactions although it is proved rigorously only for few models. The area interaction is one of such models and the complete proof of its phase transition is given here. The GNZ equations, the variational principle are discussed as well. 

In the last section, we investigate the estimation of parameters which appear in the distribution of GPP. For sake of simplicity we deal only with the activity parameter $z$ and the inverse temperature $\beta$. We present several standard procedures (MLE, MPLE, Takacs-Fiksel procedure) and a new variational procedure. We show the consistency of estimators, which highlights that many theoretical results are possible in spite of lack of explicit computations. We will see that the GNZ equations play a crucial role in this task.  For sake of simplicity the asymptotic normality is not presented but some references are given. 

Let us finish this introduction by giving standard references. Historically, the GPP have been introduced for statistical mechanics considerations and an unovoidable reference is the book by Ruelle \cite{Ruellebook}. Important theoretical contributions are also developed in two Lecture Notes \cite{GeorgiiLN, Preston} by Georgii and Preston. For the relations between GPP and stochastic geometry, we can mention the book \cite{chiu2013} by Chiu et al.  and for spatial statistic and numerical considerations, the book by M{\o}ller and Waagepetersen \cite{MW} is the standard reference. Let us mention also the book \cite{VLbook} by van Lieshout on the applications of GPP.

\tableofcontents

\section{Finite volume Gibbs point processes}
\label{Chapter1}

In this first section we present the theory of Gibbs point process on a bounded set $\L\subset \Rd$. A Gibbs point process (GPP) is a point process with interactions between the points defined via an energy functional on the space of configurations. Roughly speaking, the GPP produces random configurations for which the configurations with low energy have more chance to appear than the configurations with high energy (see Definition \ref{DefinitionFiniteVolGPP}). In Section \ref{SectionPoisson} we recall succinctly some definitions of point process theory and we introduce the reference Poisson point process. The energy functions are discussed in Section \ref{SectionEnergy} and the definiton of finite volume GPP is given in Section \ref{SectionFiniteVolumGPP}. Some first properties are presented as well. The central DLR equations and GNZ equations are treated in Sections \ref{SectionDLRFiniteVolume} and \ref{SectionGNZFiniteVolume}. Finally we finish the first section by giving Ruelle estimates in the setting of superstable and lower regular energy functions.

\subsection{Poisson point process}\label{SectionPoisson}

In this first section, we describe briefly the setting of point process theory and we introduce the reference Poisson point process. We only give the main definitions and concepts and we suggest \cite{ DVJ,MKM} for a general presentation.

The space of configurations $\C$ is defined as the set of locally finite subsets in $\Rd$:

$$ \C=\{ \g \subset \Rd, \g_\L:=\g\cap \L \text{ is finite for any bounded set }\L\subset \Rd\}.$$  

Note that we consider only the simple point configurations, which means that the points do not overlap. We denote by $\C_f$ the space of finite configurations in $\C$ and by $\C_\L$ the space of finite configurations inside $\L\subset \R^d$.

The space $\C$ is equipped with the sigma-field $\FC$ generated by the counting functions $N_\L$ for all bounded measurable $\L\subset\Rd$, where $N_\L: \g \mapsto \#\g_\L$. A point process $\G$ is then simply a measurable function from any probability space $(\Omega,\F,P)$ to $(\C,\FC)$. As usual, the distribution (or the law) of a point process $\G$ is defined by the image of $P$ to $(\C,\FC)$ by the application $\G$. We say that $\G$ has finite intensity if, for any bounded set $\L$, the expectation $\mu(\L):=E(N_\L(\G))$ is finite. In this case, $\mu$ is a sigma-finite measure called intensity measure of $\G$. When $\mu=\z\lambda^d$, where $\l^d$ is the Lebesgue measure on $\R^d$ and $\z\ge 0$ a positive real, we simply say that $\G$ has finite intensity $\z$ .

The main class of point processes is the family of Poisson point processes, which furnish the natural way of producing independent points in space. Let $\mu$ be a sigma-finite measure in $\Rd$. A Poisson point process with intensity $\mu$  is a point process  $\G$ such that, for any bounded $\L$ in $\Rd$, these properties both occur

\begin{itemize}
\item The random variable $N_\L(\G)$ is distributed following a Poisson distribution with parameter $\mu(\L)$.
\item Given the event  $\{N_\L(\G)=n\}$, the $n$ points in $\G_\L$ are independent and distributed following the distribution $\mu_\L/\mu(\L)$.
\end{itemize}

The distribution of such a Poisson point process is denoted by  $\pi^\mu$. When the intensity is $\mu=\z\l^d$, we say that the Poisson point process is stationary (or homogeneous) with intensity $\z>0$, and denote its distribution $\pi^\z$. For any measurable set $\L\subset \Rd$, we denote by $\pi_\L^\z$ the distribution of a Poisson point process with intensity $\z\l^d_\L$ which is also the  distribution of a stationary Poisson point process with intensity $\z$  restricted to $\L$. For sake of brevity, $\pi$ and $\pi_\L$ denote the distribution of Poisson point processes with intensity $\z=1$.

\subsection{Energy functions}\label{SectionEnergy}

In this section, we present the energy functions with the standard assumptions which we assume in this mini-course. The choices of energy functions come from two main motivations. First, the GPP are natural models in statistical physics for modelling continuum interacting particles systems. In general, in this setting the energy function is a sum of the energy contribution of all pairs of points (see expression \eqref{energypairwise}). The GPP are also used in spatial statistics to fit as best as possible the real datasets. So, in a first step, the energy function is chosen by the user with respect to the characteristics of the dataset. Then the parameters are estimated in a second step.

\begin{definition}\label{Energyfuntion}
An energy function is a measurable function 
 $$ H: \C_f \mapsto \R\cup\{+\infty\}$$
 such that the following assumptions hold
 
 \begin{itemize}
 \item $H$ is {\bf non-degenerate}: 
  $$H(\emptyset)<+\infty.$$
 \item H is {\bf hereditary}: for any $\g\in\C_f$ and $x\in \g$ then 
 $$H(\g)<+\infty \Rightarrow  H(\g\backslash \{x\})<+\infty.$$
 \item $H$ is {\bf stable}: there exists a constant $A$ such that for any $\g\in\C_f$
  $$H(\g)\ge A N_{\R^d}(\g).$$
 \end{itemize}  

\end{definition}

The stability implies that the energy is superlinear. If the energy function $H$ is positive then the choice $A=0$ works but in the interesting cases, the constant $A$ is negative. The hereditary means that the set of allowed configurations (configurations with finite energy) is stable when points are removed. The non-degeneracy is very natural. Without this assumption, the energy would be equal to infinity everywhere (by hereditary). \\

{\it 1) Pairwise interaction.}  Let us start with the most popular energy function which is based on a function (called pair potential) 
$$ \varphi: \R^+ \to \R\cup \{+\infty\}.$$ 
The pairwise energy function is defined for any $\g\in\C_f$ by

\begin{equation}\label{energypairwise}
H(\g)=\sum_{\{x,y\}\subset \g} \varphi(|x-y|).
\end{equation}

Note that such an energy function is trivially  hereditary and non-degenerate. The stability is more delicate and we refer to general results in \cite{Ruellebook}. However if $\varphi$ is positive the result is obvious.

A standard example coming from statistical physics is the so-called Lennard-Jones pair potential where $\varphi(r)=ar^{-12}+b r^{-6}$ with $a>0$ and $b\in \R$. In the interesting case $b<0$, the pair potential $\varphi(r)$ is positive (repulsive) for small $r$ and negative (attractive) for large $r$. The stability is not obvious and is proved in Proposition 3.2.8 in \cite{Ruellebook}.

 The Strauss interaction corresponds to the pair potential  $\varphi(r)=\1_{[0,R]}(r)$ where $R>0$ is a support parameter. This interaction exhibits a constant repulsion between the particles at distance smaller than $R$. This simple model is very popular  in spatial statistics.

The multi-Strauss interaction corresponds to the pair potential
 
$$ \varphi(r)=\sum_{i=1}^k a_i \1_{]R_{i-1},R_i]},$$
where $(a_i)_{1\le i\le k}$ is a sequence of real numbers and $0=R_0<R_1<\ldots<R_k$ a sequence of increasing real numbers. Clearly, the pair potential exhibits a constant attraction or repulsion at different scales. The stability occurs provided that the parameter $a_1$ is large enough (see Section 3.2 in \cite{Ruellebook}). \\
 
{\it 2) Energy functions coming from geometrical objects.} Several energy functions are based on local geometrical characteristics. The main motivation is to provide random configurations such that special geometrical features appear with higher probability under the Gibbs processes than the original Poisson point process. In this paragraph we give examples related to the Delaunay-Voronoi diagram. Obviously other geometrical graph structures could be considered. 

%We refer to \cite{dereudredrouilhetgeorgii} for a general presentation of the class of geometrical interactions. 

 Let us recall that for any $x\in\g\in\C_f$ the Voronoi cell $C(x,\g)$ is defined by
$$   C(x,\g)=\Big\{w\in\R^d, \text{ such that } \forall y\in\g\; |x-w|\le |x-y|\Big\}.$$  
 
The Delaunay graph with vertices $\g$  is defined by considering the edges 
 $$D(\g)=\Big\{ \{x,y\}\subset \g \text{ such that }  C(x,\g)\cap C(y,\g)\neq \emptyset\Big\}.$$
See \cite{MollerLN} for a general presentation on the Delauany-Voronoi tessellations. 

A first geometric energy function can be defined by

\begin{equation}\label{energyvoronoi}
 H(\g)=\sum_{x\in\g} \1_{C(x,\g) \text{ is bounded }}\varphi(C(x,\g)),
 \end{equation}
where $\varphi$ is any function from the space of polytopes in $\Rd$ to $\R$. Examples of such functions $\varphi$ are the Area, the $(d-1)$-Hausdorff measure of the boundary, the number of faces, etc... Clearly these energy functions are non-degenerate and hereditary. The stability holds as soon as the function $\varphi$ is bounded from below.

Another kind of geometric energy function can be constructed via a pairwise interaction along the edges of the Delaunay graph. Let us consider a finite pair potential $ \varphi: \R^+ \mapsto \R$. Then the energy function is defined by

\begin{equation}\label{energyDelaunay}
H(\g)=\sum_{\{x,y\}\subset D(\g)} \varphi(|x-y|)
\end{equation}

which is again clearly non-degenerate and hereditary. The stability occurs in dimension $d=2$ thanks to Euler's formula. Indeed  the number of edges in the Delaunay graph is linear with respect to the number of vertices. Therefore the energy function is stable as soon as the pair potential $\varphi$ is bounded from below. In higher dimension $d>2$, the stability is more complicated and not really understood. Obviously, if $\varphi$ is positive, the stability occurs.

Let us give a last example of geometric energy function which is not based on the Delaunay-Voronoi diagram but on a germ-grain structure. For any radius $R>0$ we define the germ-grain structure of $\g\in\C$ by
$$L_R(\g)=\bigcup_{x\in\g} B(x,R),$$
where $B(x,R)$ is the closed ball centred at $x$ with radius $R$. Several interesting energy functions are built from this germ-grain structure. First the Widom-Rowlinson interaction is simply defined by

\begin{equation}\label{EnergyArea}
 H(\g)=\text{Area}(L_R(\g)),
 \end{equation}
where the "Area" is simply the Lebesgue measure $\lambda^d$. This model is very popular since it is one of a few models for which the phase transition result is proved (see Section \ref{SectionNonUniqueness}). This energy function is sometimes called Area-interaction \cite{BVL,Widom70}. If the Area functional is replaced by any linear combination of the Minkowski functionals we obtain the Quermass interaction \cite{david}.

Another example is the random cluster interaction defined by 
\begin{equation}\label{EnergyCRCM}
 H(\g)=\text{Ncc}(L_R(\g)),
 \end{equation}
where Ncc denotes the functional which counts the number of connected components. This energy function is introduced first in \cite{CCK} for its relations with the Widom-Rowlinson model. See also \cite{DH} for a general study in the infinite volume regime.

\subsection{Finite Volume GPP}\label{SectionFiniteVolumGPP}

Let $\L\subset\Rd$ such that $0<\l^d(\L)<+\infty$.  In this section we define the finite volume GPP on $\L$ and we give its first properties.

 \begin{definition}\label{DefinitionFiniteVolGPP}  The finite volume Gibbs measure on $\L$ with activity $z>0$, inverse temperature $\beta\ge 0$ and energy function $H$ is the distribution

\begin{equation}\label{DefGPP}
P_\L^{z,\beta}=\frac{1}{Z_\L^{z,\beta}} z^{N_\L} e^{-\beta H} \pi_\L, 
\end{equation}
where $Z_\L^{z,\beta}$, called partition function, is the normalization constant $\int z^{N_\L}e^{-\beta H} d\pi_\L$. A finite volume Gibbs point process (GPP) on $\L$ with activity $z>0$, inverse temperature $\beta\ge 0$ and energy function $H$  is a point process on $\L$ with distribution $P_\L^{z,\beta}$.
\end{definition}  

Note that $P_\L^{z,\beta}$ is well-defined since the partition function $Z_\L^{z,\beta}$ is positive and finite. Indeed, thanks to the non degeneracy of $H$ 
$$Z_\L^{z,\beta}\ge \pi_\L({\emptyset})e^{-\beta H(\{\emptyset\})}=e^{-\l^d(\L)}e^{-\beta H(\{\emptyset\})}>0$$

 and thanks to the stability of $H$

$$ Z_\L^{z,\beta} \le e^{-\l^d(\L)} \sum_{n=0}^{+\infty} \frac{(ze^{-\beta A}\l^d(\L))^n}{n!} = e^{\l^d(\L)(ze^{-\beta A}-1)}<+\infty. $$

In the case $\beta=0$, we recover that $P_\L^{z,\beta}$ is the Poisson point process $\pi_\L^z$. So the activity  parameter $z$ is the mean number of points per unit volume when the interaction is null. When the interaction is active ($\beta>0$), $P_\L^{z,\beta}$ favours the configurations with low energy and penalizes the configurations with high energy. This distortion strengthens as $\beta$ is large.
 
  There are many motivations for the exponential form of the density in \eqref{DefGPP}. Historically, it is due to the fact that the finite volume GPP solves the variational principle of statistical physics. Indeed, $P_\L^{z,\beta}$ is the unique probability measure which realizes the minimum of the free excess energy, equal to the mean energy plus the entropy. It expresses the common idea that the equilibrium states in statistical physics minimize the energy and maximize the "disorder". This result is presented in the following proposition. Recall first that the relative entropy of a probability measure $P$ on $\C_\L$ with respect to the Poisson point process $\pi_\L^\z$  is defined by 
\begin{equation}\label{entropyR}
\I(P|\pi_\L^\z)=\left\{\begin{array}{ll}
\int \log(f)dP  & \text{if } P\preccurlyeq \pi_\L^z \text{ with } f=\frac{dP}{d\pi_\L^\z}\\
+\infty & \text{ otherwise.}
\end{array}\right.
\end{equation}

\begin{proposition} [Variational Principle] \label{PropVP}
 Let $H$ be an energy function, $z>0$, $\beta\ge 0$. Then
  $$\{P_\L^{z,\beta}\}=\text{argmin}_{P\in\P_\L} \beta E_P(H)-\log(z)E_P(N_\L)+\I(P|\pi_\L),$$
where $\P_\L$ is the space of probability measures on $\C_\L$ with finite intensity  and $E_P(H)$ is the expectation of $H$ under $P$, which is always defined (maybe equal to infinity) since $H$ is stable.
\end{proposition}
\begin{proof}
First we note that

\begin{eqnarray}\label{calcul1}
& & \beta E_{P_\L^{z,\beta}}(H)-\log(z)E_{P_\L^{z,\beta}}(N_\L)+\I(P_\L^{z,\beta}|\pi_\L)\nonumber\\
& =& \beta \int H dP_\L^{z,\beta} -\log(z)E_{P_\L^{z,\beta}}(N_\L)+\int \log\left(z^{N_\L} \frac{e^{-\beta H}}{Z_\L^{z,\beta}}\right) dP_\L^{z,\beta}\nonumber\\
&=&  -\log(Z_\L^{z,\beta}).
\end{eqnarray}
This equality implies that the minimum of $\beta E_P(H)-\log(z)E_P(N_\L)+\I(P|\pi_\L)$ should be equal to $-\log(Z_\L^{z,\beta})$. So for any $P\in\P_\L$ such that $E_P(H)<+\infty$ and $\I(P|\pi_\L)<+\infty$ let us show that $\beta E_P(H)-\log(z)E_P(N_\L)+\I(P|\pi_\L)\ge -\log(Z_\L^{z,\beta})$ with equality if and only if $P=P_\L^{z,\beta}$. Let $f$ be the density of $P$ with respect to $\pi_\L$.

\begin{eqnarray*}
\log(Z_\L^{z,\beta}) & \ge & \log \left(\int_{\{f>0\}} z^{N_\L}e^{-\beta H}d\pi_\L\right)\\
 &=& \log \left(\int z^{N_\L}e^{-\beta H}f^{-1}dP\right)\\
 &\ge & \int \log \left(z^{N_\L}e^{-\beta H}f^{-1}\right)dP\\
 &=&-\beta E_P(H)-\log(z)E_P(N_\L)-\log(f)dP. 
\end{eqnarray*}
The second inequality, due to the Jensen's inequality, is an equality if and only if $z^{N_\L}e^{-\beta H}f^{-1}$ is $P$ a.s. constant which is equivalent to $P=P_\L^{z,\beta}$. The proposition is proved. 
\end{proof}

The parameters $z$ and $\beta$ allow to fit the mean number of points and the mean value of the energy under the GPP. Indeed when $z$ increases, the mean number of points increases as well and similarly when $\beta$ increases, the mean energy decreases. This phenomenon is expressed in the following proposition. The proof is a simple computation of derivatives. 

Let us note that it is not easy to tune both parameters simultaneously since the mean number of points changes when $\beta$ is modified (and vice versa). The estimation of the parameters $z$ and $\beta$ is discussed in the last Section \ref{SectionEstimation}.

\begin{proposition}
The function $ z\mapsto E_{P_\L^{z,\beta}}(N_\L)$ is continuous and differentiable, with derivative $z\mapsto Var_{P_\L^{z,\beta}}(N_\L)/z$ on $(0,+\infty)$. Similarly the function $ \beta\mapsto E_{P_\L^{z,\beta}}(H)$ is continuous and differentiable with derivative $\beta\mapsto -Var_{P_\L^{z,\beta}}(H)$ on $\R^+$.
\end{proposition}

Let us finish this section by explaining succinctly how to simulate such finite volume GPP. There are essentially two algorithms. The first one is based on a  MCMC procedure where GPP are viewed as equilibrium states of Markov chains. The simulation is obtained by letting run for a long enough time the Markov chain. The simulation is not exact and the error is essentially controlled via a monitoring approach (see \cite{MW}). The second one is a coupling from the past algorithm which provided exact simulations. However, the computation time is often very long and these algorithms are not really that used in practice (see \cite{MK}).

\subsection{DLR equations}\label{SectionDLRFiniteVolume}

The DLR equations are due to Dobrushin, Lanford and Ruelle and give the local conditional distributions of GPP in any bounded window $\Delta$ given the configuration outside $\Delta$. We need to define a family of local energy functions $(H_\Delta)_{\Delta\subset \Rd}$.

\begin{definition}\label{EnergyDelta}
For any bounded set $\Delta$ and any finite configuration $\g\in\C_f$ we define 
$$H_\D(\g) := H(\g)-H(\g_{\D^c}),$$
with the the convention $\infty-\infty=0$.
\end{definition}

The quantity $H_\D(\g)$ gives the energetic contribution of points in $\g_\D$ towards the computation of the energy of $\g$. As an example, let us compute these quantities in the setting of pairwise interaction introduced in \eqref{energypairwise}; 

$$
H_\D(\g) = \sum_{\{x,y\} \subset \g} \varphi(|x-y|) -  \sum_{\{x,y\}\subset \g_{\D^c}}  \varphi(|x-y|) =\sum_{\begin{array}{l}
\{x,y\} \subset \g\\
\{x,y\}\cap \D\neq \emptyset
\end{array}
} \varphi(|x-y|).$$

Note that $H_\D(\g)$ does not depend only on points in $\Delta$. However, trivially  we have $ H(\g)=H_\D(\g)+H(\g_{\D^c}),$ which shows that the energy of $\g$ is the sum of the energy $H_\D(\g)$ plus something which does not depends on $\g_\D$.

\begin{proposition} [DLR equations for finite volume GPP]\label{PropositionDLR}
Let $\D\subset \L$ be two bounded sets in $\Rd$ with $\l^d(\D)>0$. Then for $P_\L^{z,\beta}$-a.s.  all $\g_{\D^c}$ 
\begin{equation}\label{DLRfiniteVolume}
P_\L^{z,\beta}(d\g_\D|\g_{\D^c}) = \frac{1}{Z_{\D}^{z,\beta}(\g_{\D^c})} z^{N_\D(\g)}e^{-\beta H_\D(\g)} \pi_\D(d\g_\D),
\end{equation}
where $Z_{\D}^{z,\beta}(\g_{\D^c})$ is the normalizing constant $\int z^{N_\D(\g)}e^{-\beta H_\D(\g)} \pi_\D(d\g_\D)$. In particular the right term in \eqref{DLRfiniteVolume} does not depend on $\L$.
\end{proposition}
\begin{proof}

From the definition of $H_\D$ and the stochastic properties of the Poisson point process we have

\begin{eqnarray*}
P_\L^{z,\beta}(d\g)&=&\frac{1}{Z_\L^{z,\beta}} z^{N_\L(\g)}e^{-\beta H(\g)} \pi_\L(d\g)\\
&=& \frac{1}{Z_\L^{z,\beta}} z^{N_\D(\g)}e^{-\beta H_\Delta(\g)}z^{N_{\L\backslash\D}(\g)}e^{-\beta H(\g_{\L\backslash \Delta})} \pi_\D(d\g_\D)\pi_{\L\backslash \D}(d\g_{\L\backslash \D}).
\end{eqnarray*}
This expression ensures that the unnormalized conditional density of $P_\L^{z,\beta}(d\g_\D|\g_{\D^c})$ with respect to $\pi_\D(d\g_\D)$ is $\g_\D\mapsto z^{N_\D(\g)}e^{-\beta H_\Delta(\g)}$. The  normalization is necessary $Z_{\D}^{z,\beta}(\g_{\D^c})$ and the proposition is proved.
\end{proof}
%
%\begin{remark} (Spatial Markov property)  \label{RemarqueMarkov}
%In the case where the energy function $H$ has a finite range $R$. The DLR equations ensure that the Probability measure $P_\L^{z,\beta}$ is a spatial Markov field in the sense that for any $\Delta\subset \L$ the  conditional distribution  $P_\L^{z,\beta}(d\g_\D|\g_{\D^c})$ depends only on $\g_{\D\oplus B(0,R)\backslash \D}$. Indeed from the definition of the finite range property,  $H_\D(\g)$ depends only on $\g_{\D\oplus B(0,R)\backslash \D}$. 
%\end{remark}

The DLR equations give the local conditional marginal distributions of GPP. They are the main tool to understand the local description of $P_\L^{z,\beta}$, in particular when $\L$ is large. Note that the local marginal distributions (not conditional) are  in general not accessible. It is a difficult point of the theory of GPP. This fact will be reinforced in the infinite volume regime, where the local distributions can be non-unique. 

 The DLR equations have a major issue due the the intractable normalization constant $Z_{\D}^{z,\beta}(\g_{\D^c})$. In the next section the problem is partially solved via the GNZ equations.

\subsection{GNZ equations}\label{SectionGNZFiniteVolume}

The GNZ equations are due to Georgii, Nguyen and Zessin and have been introduced first in \cite{NZ}. They generalize the Slivnyak-Mecke formulas for Poisson point processes. In this section we present and prove these equations. We need first to define the energy of a point inside a configuration.

\begin{definition}\label{localenergy} Let $\g\in\C_f$ be a finite configuration and $x\in\Rd$. Then the local energy of $x$ in $\g$ is defined by
$$ h(x,\g)=H(\{x\}\cup\g)-H(\g),$$
\end{definition}
with the convention $+\infty-(+\infty)=0$.
Note that if $x\in\g$ then $h(x,\g)=0$.

\begin{proposition}[GNZ equations]\label{PropGNZfiniteVolume}
For any positive measurable function $f$ from $\Rd\times\C_f$ to $\R$,

\begin{equation}\label{GNZFV}
\int \sum_{x\in\g} f(x,\g\backslash \{x\}) P_\L^{z,\beta}(d\g) =z \int \int_\L f(x,\g) e^{-\beta h(x,\g)} dxP_\L^{z,\beta}(d\g).
\end{equation}
\end{proposition}
 
\begin{proof}

Let us decompose the left term in \eqref{GNZFV}.

\begin{eqnarray*}
& & \int \sum_{x\in\g} f(x,\g\backslash \{x\}) P_\L^{z,\beta}(d\g)\\
&=& \frac{1}{Z_\L^{z,\beta}} \int \sum_{x\in\g} f(x,\g\backslash \{x\}) z^{N_\L(\g)}e^{-\beta H(\g)} \pi_\L(d\g)\\
&=& \frac{e^{-\l^d(\L)}}{Z_\L^{z,\beta}} \sum_{n=1}^{+\infty} \frac{z^n}{n!}  \sum_{k=1}^n \int_{\L^k} f(x_k,\{x_1,\ldots, x_n\}\backslash \{x_k\}) e^{-\beta H(\{x_1,\ldots, x_n\})}  dx_1\ldots dx_n\\
&=& \frac{e^{-\l^d(\L)}}{Z_\L^{z,\beta}} \sum_{n=1}^{+\infty} \frac{z^n}{(n-1)!} \int_{\L^k} f(x,\{x_1,\ldots, x_{n-1}\}) e^{-\beta H(\{x_1,\ldots, x_{n-1}\})}\\
& & \qquad \qquad \qquad \qquad \qquad e^{-\beta h(x,\{x_1,\ldots, x_{n-1}\})}  dx_1\ldots dx_{n-1} dx\\
&=& \frac{z}{Z_\L^{z,\beta}} \int_\L \int f(x,\g)z^{N_\L(\g)} e^{-\beta H(\g)}e^{-\beta h(x,\g)} \pi_\L(d\g)dx\\
&=&z \int \int_\L f(x,\g) e^{-\beta h(x,\g)} dxP_\L^{z,\beta}(d\g).
\end{eqnarray*}
\end{proof}

As usual the function $f$ in $\eqref{GNZFV}$ can be chosen without a constant sign. We just need to check that both terms in $\eqref{GNZFV}$ are integrable.

In the following proposition we show that the equations GNZ $\eqref{GNZFV}$ characterize the probability measure $P_\L^{z,\beta}$.

\begin{proposition}\label{PropGNZreverse}
Let $\L\subset \Rd$ bounded such that $\l^d(\L)>0$. Let $P$ be a probability measure on $\C_\L$ such that for  any positive measurable function $f$ from $\Rd\times\C_f$ to $\R$ 

$$
\int \sum_{x\in\g} f(x,\g\backslash \{x\}) P(d\g) =z \int \int_\L f(x,\g) e^{-\beta h(x,\g)} dxP(d\g).$$
Then it holds that $P=P_\L^{z,\beta}$.
\end{proposition}
\begin{proof}
Let us consider the measure $Q=\1_{\{H<+\infty\}}z^{-N_\L}e^{\beta H} P$. Then

\begin{eqnarray*}
& &\int \sum_{x\in\g} f(x,\g\backslash \{x\}) Q(d\g) \\
&=& \int \sum_{x\in\g} f(x,\g\backslash \{x\}) \1_{\{H(\g)<+\infty\}}z^{-N_\L(\g)}e^{\beta H(\g)} P(d\g)\\
& =&  z^{-1}\int \sum_{x\in\g} f(x,\g\backslash \{x\}) \1_{\{H(\g\backslash \{x\})<+\infty\}}\1_{\{h(x,\g\backslash \{x\})<+\infty\}}\\
& & \qquad  z^{-N_\L(\g\backslash \{x\})}e^{\beta H(\g\backslash \{x\})}e^{\beta h(x,\g\backslash \{x\})}P(d\g)\\
&=& \int \int_\L f(x,\g)  \1_{\{H(\g)<+\infty\}}\1_{\{h(x,\g)<+\infty\}} e^{-\beta h(x,\g)}z^{-N_\L(\g)}e^{\beta H(\g)}e^{\beta h(x,\g)} dxP(d\g)\\
& =&  \int \int_\L f(x,\g) \1_{\{h(x,\g)<+\infty\}} dxQ(d\g).
\end{eqnarray*}

We deduce that $Q$ satisfies the Slivnyak-Mecke formula on  $\{\g\in\C_\L, H(\g)<+\infty\}$. It is well-known (see \cite{MKM} for instance) that it implies that the measure $Q$ (after normalization) is the Poisson point process $\pi_\L$ restricted to $\{\g\in\C_\L, H(\g)<+\infty\}$. The proposition is proved.

\end{proof}

These last two propositions show that the GNZ equations contain completely the informations on $P_\L^{z,\beta}$. Note again that the normalization constant $Z_\L^{z,\beta}$ is not present in the equations.

\subsection{Ruelle estimates}\label{SectionRuelleestimates}

In this section we present Ruelle estimates in the context of superstable and lower regular energy functions. These estimates are technical and we refer to the original paper \cite{Ruelle70} for the proofs.

\begin{definition}\label{superstabilty}
 An energy function $H$ is said superstable if $H=H_1+H_2$ where $H_1$ is an energy function (see Definition \eqref{Energyfuntion}) and $H_2$ is a pairwise energy function defined in $\eqref{energypairwise}$  with a non-negative continuous pair potential $\varphi$ such that $\varphi(0)>0$. The energy function $H$ is said lower regular if there exists a summable decreasing sequence of positive reals $(\psi_k)_{k\ge 0}$ (i.e. $\sum_{k=0}^{+\infty} \psi_k<+\infty $) such that for any finite configurations $\g^1$ and $\g^2$

 \begin{equation} \label{lowerregular}
H(\g^1\cup\g^2)-H(\g^1)-H(\g^2)\ge -\sum_{k,k'\in\Z^d} \psi_{\Vert k-k'\Vert}\Big(N^2_{[k+[0,1]^d]}(\g^1)+N^2_{[k'+[0,1]^d]}(\g^2)\Big).
\end{equation}
\end{definition}

Let us give the main example of superstable and lower regular energy function.

\begin{proposition} [Proposition 1.3 \cite{Ruelle70}]  Let $H$ be a pairwise energy function with a pair potential $\varphi=\varphi_1+\varphi_2$ where $\varphi_1$ is stable and $\varphi_2$ is non-negative continuous with $\varphi_2(0)>0$. Moreover, we assume that there exists a positive decreasing function $\psi$ from $\R^+$ to $\R$ such that 
$$\int_0^{+\infty} r^{d-1}\psi(r)dr<+\infty$$

and such that for any $x\in\R$, $\varphi(x)\ge -\psi(\Vert x\Vert)$.
Then the energy function $H$ is superstable and lower regular. 
\end{proposition}

In particular, the Lennard-Jones pair potential or the Strauss pair potential defined in Section \ref{SectionEnergy} are superstable and lower regular. Note also that all geometric energy functions presented in Section \ref{SectionEnergy} are not superstable.  

\begin{proposition}[corollary 2.9 \cite{Ruelle70}]
Let $H$ be a superstable and lower regular energy function. Let $z>0$ and $\beta>0$ be fixed. Then for any bounded subset $\Delta\subset \R^d$ with $\lambda^d(\Delta)>0$ there exist two positive constants $c_1,c_2$ such that for any bounded set $\Lambda$ and $k\ge 0$

\begin{equation}\label{RuelleEstimates}
P_\L^{z,\beta} (N_\D\ge k)\le c_1e^{-c_2k^2}.
\end{equation}
\end{proposition}

In particular, Ruelle estimates \eqref{RuelleEstimates} ensure that  the random variable $N_\D$ admits exponential moments for all orders under $P_\L^{z,\beta}$. Surprisingly, the variate  $N^2_\D$ admits exponential moments for small orders. This last fact is not true under the Poisson point process $\pi_\L^z=P_\L^{z,0}$. The interaction between the points improves the integrability properties of the GPP with respect to the Poisson point process.

\section{Infinite volume Gibbs point processes}

In this section we present the theory of infinite volume GPP corresponding to the case "$\L=\Rd$" of the previous section.  Obviously, a definition inspired by \eqref{DefGPP} does not work since the energy of an infinite configuration $\g$ is meaningless. A natural construction would be to consider a sequence of finite volume GPP $(P_{\L_n}^{z,\beta})_{n\ge 1}$ on  bounded windows $\L_n=[-n,n]^d$ and let $n$ tend to infinity. It is more or less what we do in the following Sections \ref{SectionTension} and \ref{SectionAccu}, except that the convergence occurs only for a subsequence and that the field is stationarized (see equation \eqref{EmpirocalField}). As far as we know, there does not exist a general proof of the convergence of the sequence $(P_{\L_n}^{z,\beta})_{n\ge 1}$ without extracted a subsequence. The stationarization is a convenient setting here in order to use the tightness entropy tools. In Sections \ref{SectionFiniteRange} and \ref{SectionDLRequations} we prove that the accumulation points $P^{z,\beta}$ satisfy the DLR equations which is the standard definition of infinite volume GPP (see Definition \ref{DefinitioninfiniteVolumeGPP}). We make precise that the main new assumption in this section is the finite range property (see Definition \ref{FiniteRange}). It means that the points interact with each other only if their distance is smaller than a fixed constant $R>0$. The GNZ equations in the infinite volume regime are discussed in Section \ref{SectionGNZinfinite volume}. The varitional characterisation of GPP, in the spirit of Proposition \ref{PropVP}, is presented in Section \ref{SectionVariationalPrinciple}. Uniqueness and non-uniqueness results of infinite volume GPP are treated in Sections \ref{SectionUniqueness} and \ref{SectionNonUniqueness}. These results, whose proofs are completely self contained here,  ensure the existence of a phase transition for the Area energy function presented in \eqref{EnergyArea}. It means that the associated infinite volume Gibbs measures are unique for some parameters $(z,\beta)$ and non-unique for other parameters.

\subsection{The local convergence setting}\label{SectionTension}

In this section we define the topology of local convergence which is the setting we use to prove the existence of an accumulation point for the sequence of finite volume Gibbs measures. 

First, we say that a function from $\C$ to $\R$ is local if there exists a bounded set $\D\subset\Rd$ such that for all $\g\in\C$, $f(\g)=f(\g_\D)$.

\begin{definition} The local convergence topology on the space of probability measures on $\C$ is the smallest topology such that for any local bounded function $f$ from $\C$ to $\R$ the function $P\mapsto \int f dP$ is continuous. We denote by $\tL$ this topology.
\end{definition}

Let us note that the continuity of functions $f$ in the previous definition is not required. For instance the function $\g\mapsto f(\g)=\1_{N_\D(\g)\ge k}$, where $\D$ is a bounded set in $\Rd$ and $k$ any integer, is a bounded local function. For any vector $u\in\Rd$ we denote by $\tau_u$ the translation by the vector $u$ acting on $\Rd$ or $\C$. A probability $P$ on $\C$ is said stationary (or shift invariant) if for any vector $u\in\Rd$ $P=P\circ \tau_u^{-1}$.

Our tightness tool is based on the specific entropy which is defined for any stationary probability $P$ on $\C$ by

\begin{equation}\label{entropyspecific}
 \I_\z(P)=\lim_{n\to+\infty} \frac{1}{\l^d(\L_n)} \I(P_{\L_n}|\pi_{\L_n}^\z),
 \end{equation}

 where $\I(P_{\L_n}|\pi_{\L_n}^\z)$ is the relative entropy of $P_{\L_n}$, the projection of $P$ on $\L_n$, with respect to $\pi_{\L_n}^\z$ (see Definition \eqref{entropyR}). Note that the specific entropy $\I_\z(P)$ always exists (i.e. the limit in \eqref{entropyspecific} exists); see chapter 15 in \cite{GeorgiiBook}. The tightness tool presented in Lemma \ref{LemmaTightness} below is a consequence of the following proposition. 

\begin{proposition}[Proposition 15.14 \cite{GeorgiiBook}]\label{tensionG} 
For any $\z>0$ and any value $K\ge 0$, the set

$$ \{ P\in\P \text{ such that } \I_\z(P)\le K\} $$ 

is sequentially compact for the topology $\tL$, where $\P$ is the space of stationary probability measures on $\C$ with finite intensity.
\end{proposition}

\subsection{An accumulation point $P^{z,\beta}$} \label{SectionAccu}

In this section we prove the existence of an accumulation point for a sequence of stationarized finite volume GPP. To the end we consider the Gibbs measures $(P_{\L_n}^{z,\beta})_{n\ge 1}$ on $\L_n:=[-n,n]^d$, where $(P_\L^{z,\beta})$ is defined in \eqref{DefGPP} for any $z>0$, $\beta\ge 0$ and energy function $H$. We assume that $H$ is {\bf stationary}, which means that for any vector $u\in\Rd$  and any finite configuration $\g\in\C_f$ 

$$ H(\tau_u(\g))=H(\g).$$

For any $n\ge 1$, the empirical field $\bar P_{\L_n}^{z,\beta}$ is defined by the probability measure on $\C$ such that for any test function $f$

\begin{equation}\label{EmpirocalField}
\int f(\g)\bar P_{\L_n}^{z,\beta}(d\g)= \frac{1}{\l^d(\L_n)} \int_ {\L_n} \int f(\tau_u(\g)) P_{\L_n}^{z,\beta}(d\g)du.
\end{equation}

The probability measure $\bar P_{\L_n}^{z,\beta}$ can be interpreted as the Gibbs measure $P_{\L_n}^{z,\beta}$ where the origin of the space $\Rd$ (i.e. the point $\{0\}$) is replaced by a random point chosen uniformly inside $\L_n$. It is a kind of stationarization of $P_{\L_n}^{z,\beta}$ and any accumulation point of the sequence $(\bar P_{\L_n}^{z,\beta})_{n\ge 1}$is necessary stationary.

\begin{proposition}\label{existenceAccu}
The sequence $(\bar P_{\L_n}^{z,\beta})_{n\ge 1}$ is tight for the $\tL$ topology. We denote by $P^{z,\beta}$ any of its accumulation points.
\end{proposition}

\begin{proof}

Our tightness tool is the following lemma whose the proof is a consequence of Proposition \ref{tensionG} (See also Proposition 15.52 in \cite{GeorgiiBook}).

\begin{lemma}\label{LemmaTightness} The sequence $(\bar P_{\L_n}^{z,\beta})_{n\ge 1}$ is tight for the $\tL$ topology if there exits $\z>0$ such that
\begin{equation}\label{UnifBound}
 \sup_{n\ge 1} \frac{1}{\l^d(\L_n)} \I(P_{\L_n}^{z,\beta}|\pi_{\L_n}^{\z}) <+\infty.
 \end{equation}
\end{lemma}

So, let us compute $\I(P_{\L_n}^{z,\beta}|\pi_{\L_n}^{\z})$ and check that we can find $\z>0$ such that \eqref{UnifBound} holds.

\begin{eqnarray*}
\I(P_{\L_n}^{z,\beta}|\pi_{\L_n}^{\z}) & =& \int \log\left(\frac{dP_{\L_n}^{z,\beta}}{d\pi_{\L_n}^{\z}}\right) dP_{\L_n}^{z,\beta}\\
& =& \int\left[ \log\left(\frac{dP_{\L_n}^{z,\beta}}{d\pi_{\L_n}}\right)+\log\left(\frac{d\pi_{\L_n}}{d\pi_{\L_n}^{\z}}\right)\right] dP_{\L_n}^{z,\beta}\\
& =& \int \left[ \log\left(z^{N_{\L_n}}\frac{e^{-\beta H}}{Z_{\L_n}^{z,\beta}}\right) + \log\left(e^{(\z-1)\l^d(\L_n)} \left(\frac{1}{\z}\right)^{N_{\L_n}}\right) \right]  dP_{\L_n}^{z,\beta} \\
& = & \int \left[-\beta H+\log\left(\frac{z}{\z}\right)N_{\L_n} \right]dP_{\L_n}^{z,\beta}+(\z-1)\l^d(\L_n)-\log(Z_{\L_n}^{z,\beta}).
\end{eqnarray*}

Thanks to the non degeneracy and the stability of $H$ we find that

\begin{eqnarray*}
\I(P_{\L_n}^{z,\beta}|\pi_{\L_n}^{\z}) & \le &  \int \left(-A\beta+\log\left(\frac{z}{\z}\right)\right)N_{\L_n} dP_{\L_n}^{z,\beta}+\l^d(\L_n)\Big((\z-1)+1+\beta H(\{\emptyset\})\Big).
\end{eqnarray*}
Choosing $\z>0$ such that $-A\beta+\log(z/\z)\le 0$ we obtain
 \begin{eqnarray*}
\I(P_{\L_n}^{z,\beta}|\pi_{\L_n}^{\z}) & \le & \l^d(\L_n)(\z+\beta H(\{\emptyset\})
\end{eqnarray*}
and \eqref{UnifBound} holds. Proposition \ref{existenceAccu} is proved.
\end{proof}

In the following, for sake of simplicity, we  say that $\bar P_{\L_n}^{z,\beta}$ converges to $P^{z,\beta}$ although it occurs only for a subsequence. 

Note that the existence of an accumulation points holds under very weak assumptions on the energy function $H$. Indeed the two major assumptions are the stability and the stationarity. The superstability or the lower regularity presented in Definition \ref{superstabilty} are not required here. However, if the energy function $H$ is superstable and lower regular, then  the accumularion points $P^{z,\beta}$ inherits Ruelle estimates \eqref{RuelleEstimates}. This fact is obvious since the function $\g\mapsto \1_{\{N_\D(\g)\ge k\}}$ is locally bounded.

\begin{corollary}
Let $H$ be a superstable and lower regular energy function (see Definition \ref{superstabilty}). Let $z>0$ and $\beta>0$ be fixed. Then for any bounded subset $\Delta\subset \R^d$ with $\lambda^d(\Delta)>0$, there exists $c_1$ and $c_2$ two positive constants such that for any $k\ge 0$

\begin{equation}\label{RuelleEstimatesinfinitevolume}
P^{z,\beta} (N_\D\ge k)\le c_1e^{-c_2k^2}.
\end{equation}
\end{corollary}

The important point now is to prove that $P^{z,\beta}$ satisfies good stochastic properties as for instance the DLR or GNZ equations. At this stage, without extra assumptions, these equations are not necessarily satisfied. Indeed it is possible to build energy functions $H$ such that the accumulation  point $P^{z,\beta}$ is degenerated and charges only the empty configuration. In this mini-course our extra assumption is the finite range property presented in the following section. More general settings have been investigated for instance in \cite{dereudredrouilhetgeorgii} or  \cite{Ruellebook}.

\subsection{The finite range property} \label{SectionFiniteRange}

The finite range property expresses that further a certain distance distance $R>0$ the points do not interact each other. Let us recall the Minkoswki $\oplus$ operator acting on sets in $\R^d$. For any two sets $A,B\subset \R^d$, the set $A\oplus B$ is defined by  $\{x+y, x\in A \text{ and } y\in B\}$.

\begin{definition}\label{FiniteRange}

The energy function $H$ has a finite range $R>0$ if for every bounded $\D$, the local energy $H_\D$ (see Definition \ref{EnergyDelta}) is a local function on $\D\oplus B(0,R)$. It means that for any finite configuration $\g\in\C_f$ 

$$H_\D(\g) := H(\g)-H(\g_{\D^c})=H(\g_{\D\oplus B(0,R)})-H(\g_{\D\oplus B(0,R)\backslash\D^c}).$$

\end{definition}

 Let us illustrate the finite range property in the setting of pairwise interaction defined in \eqref{energypairwise}. Assume that the interaction potential $\varphi :\R^+ \to \R\cup\{+\infty\}$ has a support included in $[0,R]$. Then the associated energy function has a finite $R$;

\begin{eqnarray*}
H_\D(\g)
& =& \sum_{\begin{array}{l}
\{x,y\} \subset \g\\
\{x,y\}\cap \D\neq \emptyset\\
|x-y|\le R
\end{array}
} \varphi(|x-y|)\\
& =& \sum_{\begin{array}{l}
\{x,y\} \subset \g_{\D\oplus B(0,R)}\\
\{x,y\}\cap \D\neq \emptyset\\
\end{array}
} \varphi(|x-y|).
\end{eqnarray*}

Also the area energy function \eqref{EnergyArea} inherits the finite range property. A simple computation gives

\begin{equation}
H_\D(\g) = \text{Area}\left(\bigcup_{x\in\g_{\D}} B(x,R) \backslash \bigcup_{x\in\g_{\D\oplus B(0,2R) \backslash \D}} B(x,R)\right)
 \end{equation}
 which provides a range of interaction equals to $2R$.

Let us note that the energy functions defined in \eqref{energyvoronoi},\eqref{energyDelaunay} and \eqref{EnergyCRCM}  do not have the finite range property. Similarly the pairwise energy function \eqref{energypairwise} with the Lennard-Jones potential is not finite range since the support of the pair potential is not bounded. A truncated version of such potential is sometimes considered.

Let us finish this section by noting that the finite range property allows  to extend the domain of definition of  $H_\D$ from the space $\C_f$ to the set $\C$. Indeed, since $H_\D(\g)=H_\D(\g_{\D\oplus B(0,R)})$, this equality provides a definition of  $H_\D(\g)$ when $\g$ is in $\C$. This point is crucial in order to correctly define the DLR equations in the infinite volume regime.

\subsection{DLR equations}\label{SectionDLRequations}

In section 1 on the finite volume GPP, the DLR equations are presented as properties for $P_\L^{z,\beta}$ (see Section \ref{SectionDLRFiniteVolume}). In the setting of infinite volume GPP, the DLR equations are the main points of  the definition of GPP. 

\begin{definition}[infinite volume GPP] \label{DefinitioninfiniteVolumeGPP} Let $H$ be a stationary and finite range energy function. A stationary probability $P$ on $\C$ is an infinite volume Gibbs measure with activity $z>0$, inverse temperature $\beta\ge 0$ and energy function $H$ if for any bounded $\Delta\subset \Rd$ such that $\l^d(\D)>0$ then for $P$-a.s. all $\g_{\D^c}$
\begin{equation}\label{DLRinfiniteVolume}
P(d\g_\D|\g_{\D^c}) = \frac{1}{Z_{\D}^{z,\beta}(\g_{\D^c})} z^{N_\D(\g)}e^{-\beta H_\D(\g)} \pi_\D(d\g_\D),
\end{equation}
where $Z_{\D}^{z,\beta}(\g_{\D^c})$ is the normalizing constant $\int z^{N_\D(\g)}e^{-\beta H_\D(\g)} \pi_\D(d\g_\D)$. As usual, an infinite volume GPP is a point process whose distribution is an infinite volume Gibbs measure.
 \end{definition}

Note that the DLR equations \eqref{DLRinfiniteVolume} make sense since $H_\D(\g)$ is well defined for any configuration $\g\in\C$ (see the end of Section \ref{SectionFiniteRange}). Note also that the DLR equations \eqref{DLRinfiniteVolume} can be reformulated in an integral form. Indeed $P$ satisfies \eqref{DLRinfiniteVolume} if and only if for any local bounded function $f$ from $\C$ to $\R$

\begin{equation}\label{DLRintegral}
\int fdP=\int f(\g'_{\D}\cup \g_{\D^c}) \frac{1}{Z_{\D}^{z,\beta}(\g_{\D^c})} z^{N_\D(\g'_\D)}e^{-\beta H_\D(\g'_{\D}\cup \g_{\D^c})} \pi_\D(d\g'_\D)   P(d\g).
\end{equation}

The term "equation" is now highlighted by the formulation \eqref{DLRintegral} since the unknown variate $P$ appears in both left and right sides. The existence, uniqueness and non-uniqueness of solutions  of such DLR equations are non trivial questions. In the next theorem, we show that the accumulation point $P^{z,\beta}$ obtained in Section \ref{SectionAccu} is such a solution. Infinite volume Gibbs measure exist and the question of existence is solved. The uniqueness and non-uniqueness are discussed in Sections \ref{SectionUniqueness} and \ref{SectionNonUniqueness}.

\begin{theorem} \label{THexistence}
Let $H$ be a stationary and finite range energy function. Then for any $z>0$ and $\beta\ge 0$ the probability measure $P^{z,\beta}$ defined in Proposition \ref{existenceAccu} is an infinite volume Gibbs measure. 
\end{theorem}

\begin{proof}
We have just to check that $P^{z,\beta}$ satisfies, for any bounded $\Delta$  and any positive local bounded function $f$, the equation \eqref{DLRintegral}. Let us define the function $f_\Delta$ by
$$ f_\Delta: \g \mapsto \int f(\g'_{\D}\cup \g_{\D^c}) \frac{1}{Z_{\D}^{z,\beta}(\g_{\D^c})} z^{N_\D(\g'_\D)}e^{-\beta H_\D(\g'_{\D}\cup \g_{\D^c})} \pi_\D(d\g'_\D).$$
Since $f$ is local and bounded and since $H$ is finite range, the function $f_\D$ is bounded and local as well. From the convergence of the sequence  $(\bar P_{\L_n}^{z,\beta})_{n\ge 1}$ to $P^{z,\beta}$ with respect to the $\tL$ topology, we have
\begin{eqnarray}\label{calcul1}
\int f_\D dP^{z,\beta} & = & \lim_{n\to\infty} \int f_\D d\bar P_{\L_n}^{z,\beta}\nonumber\\
&=& \lim_{n\to\infty} \frac{1}{\l^d(\L_n)} \int_{\L_n} \int f_\D(\tau_u(\g)) P_{\L_n}^{z,\beta}(d\g)du.\nonumber\\
& = & \lim_{n\to\infty} \frac{1}{\l^d(\L_n)} \int_{\L_n} \int \int f(\g'_{\D}\cup \tau_u(\g)_{\D^c}) \frac{z^{N_\D(\g'_\D)}}{Z_{\D}^{z,\beta}(\tau_u(\g)_{\D^c})} e^{-\beta H_\D(\g'_{\D}\cup \tau_u(\g)_{\D^c})}\nonumber\\
& &\qquad \qquad\qquad\qquad \pi_\D(d\g'_\D) P_{\L_n}^{z,\beta}(d\g)du\nonumber\\
& =& \lim_{n\to\infty} \frac{1}{\l^d(\L_n)} \int_{\L_n} \int \int f\Big(\tau_u\big(\g'_{\tau_{-u}(\Delta)}\cup \g_{\tau_{-u}(\Delta)^c}\big)\Big)\frac{z^{N_{\tau_{-u}(\D)}(\g'_{\tau_{-u}(\Delta)})}}{Z^{z,\beta}_{\tau_{-u}(\Delta)}(\g_{\tau_{-u}(\Delta)^c})}\nonumber\\
& &  e^{-\beta H_{\tau_{-u}(\D)}\big(\g'_{\tau_{-u}(\Delta)}\cup \g_{\tau_{-u}(\Delta)^c}\big)} \pi_{\tau_{-u}(\Delta)}(d\g'_{\tau_{-u}(\Delta)}) P_{\L_n}^{z,\beta}(d\g)du.
\end{eqnarray}

Denoting by $\L_n^*$ the set of $u\in\L_n$ such that $\tau_{-u}(\Delta)\subset \L_n$, by Proposition \ref{PropositionDLR}, $P_{\L_n}^{z,\beta}$ satisfies the DLR equation on $\tau_{-u}(\Delta)$ as soon as $\tau_{-u}(\Delta)\subset \L_n$ (i.e. $u\in\L_n^*$). It follows that for any $u\in\L_n^*$

\begin{eqnarray}\label{Calcul2}
& & \int f(\tau_u\g) P_{\L_n}^{z,\beta}(d\g)\nonumber\\
&= & \int \int f\Big(\tau_u\big(\g'_{\tau_{-u}(\Delta)}\cup \g_{\tau_{-u}(\Delta)^c}\big)\Big)\frac{z^{N_{\tau_{-u}(\D)}(\g'_{\tau_{-u}(\Delta)})}}{Z^{z,\beta}_{\tau_{-u}(\Delta)}(\g_{\tau_{-u}(\Delta)^c})}  e^{-\beta H_{\tau_{-u}(\D)}\big(\g'_{\tau_{-u}(\Delta)}\cup \g_{\tau_{-u}(\Delta)^c}\big)}\nonumber\\
& & \qquad \qquad \qquad  \pi_{\tau_{-u}(\Delta)}(d\g'_{\tau_{-u}(\Delta)}) P_{\L_n}^{z,\beta}(d\g).
\end{eqnarray}

By noting that $\l^d(\L_n^*)$ is equivalent to $\l^d(\L_n)$ when $n$ goes to infinity, we obtain in compiling \eqref{calcul1} and \eqref{Calcul2}

\begin{eqnarray*}
\int f_\D dP^{z,\beta} & = & \lim_{n\to\infty} \frac{1}{\l^d(\L_n)} \int_{\L^*_n} \int \int f(\tau_u\g) P_{\L_n}^{z,\beta}(d\g)du\\
& = & \lim_{n\to\infty} \int f(\g) \bar P_{\L_n}^{z,\beta}(d\g)\\
& =& \int fdP^{z,\beta}
\end{eqnarray*}
which gives the expected integral DLR equation on $\D$ with test function $f$.

\end{proof}

\subsection{GNZ equations} \label{SectionGNZinfinite volume}
 
In this section we deal with the GNZ equations in the infinite volume regime. As in the finite volume case, the main advantage of such equations is that the intractable normalization factor $Z_\L^{z,\beta}$ is not present.

Note first that, in the setting of finite range interaction $R>0$, the local energy $h(x,\g)$ defined in Definition \ref{localenergy} is well-defined for any configuration $\g\in\C$ even if $\g$ is infinite. Indeed, we clearly have $ h(x,\g)=h(x,\g_{B(x,R)})$.

\begin{theorem}\label{PropositionGNZinfiniteVolume}
Let $P$ be a probability measure on $\C$. Let $H$ be a finite range energy function and $z>0$, $\beta\ge 0$ be two parameters. Then $P$ is an infinite volume Gibbs measure with energy function $H$, activity $z>0$ and inverse temperature $\beta$ if and only if for any positive measurable function $f$ from $\Rd\times\C$ to $\R$
\begin{equation}\label{GNZequtaions}
\int \sum_{x\in\g} f(x,\g\backslash \{x\}) P(d\g) =z \int \int_{\Rd} f(x,\g) e^{-\beta h(x,\g)} dxP(d\g).
\end{equation}
\end{theorem}

\begin{proof}
Let us start with the proof of the "only if" part. Let $P$ be an infinite volume Gibbs measure. By standard  monotonicity arguments it is sufficient to prove \eqref{GNZequtaions} for any local positive measurable function $f$. So let $\Delta\subset \Rd$ be a bounded set such that $f(x,\g)=1_\D(x)f(x,\g_\D)$. Applying now the DLR equation \eqref{DLRintegral} on the set $\Delta$ we find 

\begin{eqnarray*}
& & \int \sum_{x\in\g} f(x,\g\backslash \{x\}) P(d\g)\\
 & = & \int\int \sum_{x\in\g'_\D} f(x,\g'_\D\backslash \{x\}) \frac{1}{Z_{\D}^{z,\beta}(\g_{\D^c})} z^{N_{\D}(\g'_\D)}e^{-\beta H_\D(\g'_{\D}\cup \g_{\D^c})} \pi_\D(d\g'_\D)   P(d\g). \\
\end{eqnarray*}
 
 By computations similar to those developed in the proof of Proposition \ref{PropGNZfiniteVolume}, we obtain

\begin{eqnarray*}
\int \sum_{x\in\g} f(x,\g\backslash \{x\}) P(d\g) & = & z \int \int_\D \int f(x,\g'_\D) \frac{1}{Z_{\D}^{z,\beta}(\g_{\D^c})} e^{-\beta h(x,\g'_{\D}\cup \g_{\D^c})}\\
& & \qquad  z^{N_{\D}(\g'_\D)}e^{-\beta H_\D(\g'_{\D}\cup \g_{\D^c})} \pi_\D(d\g'_\D) dx  P(d\g) \\
& = & z \int \int_{\Rd} f(x,\g) e^{-\beta h(x,\g)} dxP(d\g).
\end{eqnarray*}

Let us now turn to the "if part". Applying equation \eqref{GNZequtaions} to the function $\tilde f(x,\g)=\psi(\g_{\D^c})f(x,\g)$ where $f$ is a local positive function with support $\D$ and $\psi$ a positive test function we find

$$ \int \psi(\g_{\D^c}) \sum_{x\in\g_\Delta} f(x,\g\backslash \{x\}) P(d\g) =z \int \psi(\g_{\D^c})\int_{\Rd} f(x,\g) e^{-\beta h(x,\g)} dxP(d\g).$$

This implies that for $P$ almost all $\g_{\D^c}$ the conditional probability measure $P(d\g_\D|\g_{\D^c})$ solves the GNZ equations on $\D$ with local energy function $\g_\D\mapsto h(x,\g_\D\cup\g_{\D^c})$. Following an adaptation of the proof of Proposition \ref{PropGNZreverse}, we get that 

$$P(d\g_\D|\g_{\D^c}) = \frac{1}{Z_{\D}^{z,\beta}(\g_{\D^c})} z^{N_\D(\g)}e^{-\beta H_\D(\g)} \pi_\D(d\g_\D),$$

which is exactly the DLR equation \eqref{DLRinfiniteVolume} on $\Delta$. The theorem is proved.

\end{proof}
%
%\begin{remark} The GNZ equations \eqref{GNZFiniteVolume} provide an implicit definition of the GPP where $P$ appear in both sides right and left  of the equations. We recover that the uniqueness and non uniqueness issues of infinite GPP are not trivial. Note again that these equations are particularly simple since the normalization factors $Z_\D^{z,\beta}$  have disappear. 
%\end{remark}

Let us finish this section with an application of the GNZ equations which highlights that some properties of infinite volume GPP can be extracted from  the implicit GNZ equations.

\begin{proposition}  Let $\G$ be a infinite volume GPP for the hardcore pairwise interaction $\varphi(r)=+\infty\1_{[0,R]}(r)$ (see Definition \eqref{energypairwise}) and the activity $z>0$.  Then 

\begin{equation}\label{borneHardcore}
\frac{z}{1+z v_d R^d} \le E\left(N_{[0,1]^d}(\G)\right) \le z,
\end{equation}
where $v_d$ is the volume of the unit ball in $\Rd$.
\end{proposition}
 
Note that the inverse temperature $\beta$ does not play any role here and that $E_P\left(N_{[0,1]^d}(\G)\right)$ is simply the intensity of $\G$.
 
\begin{proof} 
The local energy of such harcore pairwise interaction is given by

$$ h(x,\g)=\sum_{y\in\g_{B(x,R)}} \varphi(|x-y|)= +\infty \1_{\g_{B(x,R)}\neq \emptyset}.$$

So the GNZ equation \eqref{GNZequtaions} with the function $f(x,\g)=\1_{[0,1]^d}(x)$ gives

$$E\left(N_{[0,1]^d}(\G)\right)=z\int_{[0,1]^d} P(\G_{B(x,R)}= \emptyset) dx=z\ P(\G_{B(0,R)}= \emptyset),$$

which provides a relation between the intensity and the spherical contact distribution of $\G$. The upper bound in 
\eqref{borneHardcore} follows. For the lower bound we have
\begin{eqnarray*}
E_P\left(N_{[0,1]^d}(\G)\right) & =& z\ P(\G_{B(0,R)}= \emptyset)\\
& \ge & z\left(1-E_P\left(N_{B(0,R)}(\G)\right)\right)\\
& = & z\left(1-v_dR^rE_P\left(N_{[0,1]^d}(\G)\right)\right). \\
\end{eqnarray*}
\end{proof}

Note also that a natural upper bound for $E_P\left(N_{[0,1]^d}\right) $ is obtained via the closed packing configuration. For instance, in dimension $d=2$, it gives the upper bound $\pi/(2\sqrt{3} R^2)$.

\subsection{Variational principle}\label{SectionVariationalPrinciple}

In this section, we extend the variational principle for finite volume GPP presented in  Proposition \ref{PropVP} to the setting of  infinite volume GPP. For brevity we present only the result without the proof which can be found in \cite{DereudreVP}.

The variational principle claims that the Gibbs measures are the minimizers of the free excess energy defined by the sum of the the mean energy and the specific entropy. Moreover, the minimum is equal to minus the pressure. Let us first define all these macroscopic  quantities.

Let us start by introducing the pressure with free boundary condition. It is defined as the following limit

\begin{equation}\label{pressure}
p^{z,\beta}:=\lim_{n\to +\infty} \frac{1}{|\L_n|} \ln (Z^{z,\beta}_{\L_n}),
\end{equation}

 The existence of such limit is proved for instance in Lemma 1 in \cite{DereudreVP}.

The second macroscopic quantity involves the mean energy of a stationary probability measure $P$. It is also defined by a limit but, in opposition to the pressure, we have to assume that it exists. The proof of such existence is generally based on stationary arguments and  nice representations of the energy contribution per unit volume. It depends strongly on the expression of the energy function $H$. Examples are given below. So for any stationary probability measure $P$ on $\C$ we assume that the following limit exists in $\R\cup\{+\infty\}$,

\begin{equation}\label{meanenergy}
H(P):=\lim_{n\to \infty} \frac{1}{|\L_n|} \int H(\g_{\L_n}) dP(\g),
\end{equation}
and we call the limit mean energy of $P$.

We need to introduce a technical assumption on the boundary effects of $H$. We assume that for any infinite volume Gibbs measure $P$ 

\begin{equation}\label{boundary}
\lim_{n\to\infty} \frac{1}{|\L_n|} \int \partial H_{\L_n}(\g)dP(\g)=0,
\end{equation}
where $\partial H_{\L_n}(\g)= H_{\L_n}(\g)-H(\g_{\L_n})$.

\begin{theorem}[Variational Principle, Theorem 1, \cite{DereudreVP}]\label{TheoremVP}

 We assume that $H$ is stationary and finite range. Moreover, we assume that the mean energy exists for any stationary probability measure $P$ (i.e. the limit \eqref{meanenergy} exists) and that the boundary effects assumption \eqref{boundary} holds. Let $z>0$ and $\beta\ge 0$ two parameters. Then for any stationary probability measure $P$ on $\C$ with finite intensity

\begin{equation}\label{inequality}
 I_1(P)+\beta H(P)-\log(z)E_P(N_{[0,1]^d}) \ge -p^{z,\beta},
\end{equation}
 
with equality if and only if $P$ is a Gibbs measure with activity $z>0$, inverse temperature $\beta$ and energy function $H$. 
\end{theorem}

Let us finish this section by presenting the two fundamental examples of energy functions satisfying the assumptions of Theorem \ref{TheoremVP}.

\begin{proposition} Let $H$ be the Area energy function defined in \eqref{EnergyArea}. Then both limits \eqref{meanenergy} and \eqref{boundary} exist. In particular, the assumptions of Theorem \ref{TheoremVP} are satisfied and the variational principle holds.
\end{proposition}

\begin{proof}
Let us prove only that the limit \eqref{meanenergy} exists. The existence of limit \eqref{boundary}  can be shown in the same way. By definition of $H$ and the stationarity of $P$,

\begin{eqnarray}\label{decompositionArea}
\int H(\g_{\L_n}) P(d\g)&=& \int \text{Area}(L_R(\g_{\L_n}))P(d\g)\nonumber\\
&=& \lambda^d(\L_n) \int \text{Area}(L_R(\g)\cap[0,1]^d)P(d\g)\nonumber\\
& & + \int \Big(\text{Area}(L_R(\g_{\L_n}))- \text{Area}(L_R(\g)\cap[-n,n]^d)\Big)P(d\g).
\end{eqnarray}
By geometric arguments, we get that 

$$ \Big|\text{Area}(L_R(\g_{\L_n}))- \text{Area}(L_R(\g)\cap[-n,n]^d) \Big| \le Cn^{d-1}, $$
for some constant $C>0$. We deduce that the limit  \eqref{meanenergy} exists with 
$$H(P)= \int \text{Area}(L_R(\g)\cap[0,1]^d)P(d\g).$$

\end{proof}

\begin{proposition} \label{PropositionVPpairwise}
Let $H$ be the pairwise energy function defined in \eqref{energypairwise} with a superstable, lower regular pair potential with compact support. Then the both limits \eqref{meanenergy} and \eqref{boundary} exist. In particular the assumptions of Theorem \ref{TheoremVP} are satisfied and the variational principle holds.
\end{proposition}

\begin{proof}

Since the potential $\varphi$ is stable with compact support, we deduce that $\varphi\ge 2A$ and $H$ is finite range and lower regular. In this setting, the existence of the limit \eqref{meanenergy} is proved in  \cite{Georgii94}, Theorem 1 with

\begin{equation}
 H(P)=\left\{ \begin{array}{ll}
\frac 12\int \sum_{0\neq x\in\g} \varphi(x)P^0(d\g) & \text{ if  } E_P(N^2_{[0,1]^d})<\infty \\
+\infty & \text{ otherwise}
\end{array}\right.
\end{equation}

where $P^0$ is the Palm measure of $P$. Recall that $P^0$ can be viewed as the natural version of the conditional probability $P(.|0\in\g)$ (see \cite{MKM} for more details). It remains to prove the existence of the limit \eqref{boundary} for any Gibbs measure $P$ on $\C$. A simple computation gives that, for any $\g\in\C$,

$$ \partial H_{\L_n}(\g)= \sum_{x\in\g_{\L_{n}^\oplus\backslash \L_n}}
\sum_{ y\in\g_{\L_n\backslash \L_n^\ominus}} \varphi(x-y),$$

where $\L_n^\oplus=\L_{n+R_0}$ and $\L_n^\ominus=\L_{n-R_0}$ with $R_0$ an integer larger than the range of the interaction $R$. 

Therefore thanks to the stationarity of $P$ and the GNZ equations \eqref{GNZequtaions}, we obtain

\begin{eqnarray*}
\left|\int\partial H_{\L_n}(\g) dP(\g)\right| & \le & \int\sum_{x\in\g_{\L_{n}^\oplus\backslash \L_n}} \sum_{y\in\g\backslash \{x\}} |\varphi(x-y)|dP(\g)\\
 & = & z\int\int_{\L^\oplus_{n}\backslash \L_n} e^{-\beta\sum_{y\in\g} \varphi(x-y)}  \sum_{y\in\g} |\phi(x-y)|dx dP(\g)\\
 & = &z|\L^\oplus_{n}\backslash \L_n| \int e^{-\beta\sum_{y\in\g_{B(0,R_0)}} \varphi(y)}  \sum_{y\in\g_{B(0,R_0)}} |\varphi(y)| dP(\g).\\
\end{eqnarray*}

Since $\varphi\ge 2A$, denoting by $C:=\sup_{c\in[2A;+\infty)} |c|e^{-\beta c}<\infty$ we find that

\begin{eqnarray}\label{bord}
\left|\int\partial H_{\L_n}(\g) dP(\g)\right|
 & \le & zC|\L^\oplus_{n}\backslash \L_n| \int N_{B(0,R_0)}(\g)e^{-2\beta A N_{B(0,R_0)}(\g)} dP(d\g).
\end{eqnarray}
Using Ruelle estimates \eqref{RuelleEstimatesinfinitevolume}, the integral in the right term of \eqref{bord} is finite. The boundary assumption \eqref{boundary} follows.

\end{proof}

\subsection{A uniqueness result}\label{SectionUniqueness}

In this section we investigate the uniqueness of infinite volume Gibbs measures. The common belief claims that the Gibbs measures are unique when the activity $z$ or (and) the inverse temperature $\beta$ are small enough (low activity, high temperature regime). The non-uniqueness phenomenon (discussed in the next section) are in general related to some issues with the energy part in the variational principle (see Theorem \ref{TheoremVP}). Indeed, either the mean energy has several minimizers or there is a conflict between the energy and the entropy. Therefore it is natural to expect the the Gibbs measures are unique when $\beta$ is small enough. When $z$ is small, the mean number of points per unit volume is low and so the energy is in general low as well. 

As far as we know, there do not exist general results which prove the uniqueness for small $\beta$ or small $z$. In the case of pairwise energy functions \eqref{energypairwise}, the uniqueness for any $\beta>0$ and $z>0$ small enough is proved  via the Kirkwood-Salsburg equations (see Theorem 5.7 \cite{Ruelle70}). An extension of the Dobrushin uniqueness criterium in the continuum is developed as well \cite{DobPer}.
The uniqueness of GPP can also be obtained via the cluster expansion machinery which provides a power series expansion of the partition function when $z$ and $\beta$ are small enough. This approach has been introduced first by Mayer and Montroll \cite{MayerMontroll} and we refer to \cite{PU} for a general presentation.

In this section we give a simple and self-contained proof of the uniqueness of GPP for all $\beta\ge 0$ and  any  $z>0$ small enough. We just assume that the energy function $H$ has a local energy $h$ uniformly bounded from below. This setting covers for instance the case of pairwise energy function \eqref{energypairwise} with non-negative pair potential or the Area energy function \eqref{EnergyArea}.  
  
Let us start by recalling the existence of a percolation threshold for the Poisson Boolean model. For any configuration $\g\in\C$ the percolation of $L_R(\g)=\cup_{x\in\g} B(x,R)$ means the existence of an unbounded connected component in $L_R(\g)$ .

\begin{proposition}[Theorem 1 \cite{Hall}] \label{PropoPerco} For any $d\ge 2$, there exists $0<z_d<+\infty$ such that for $z<z_d$,  $\pi^z(L_{1/2}\text{ percolates })=0$ and for $z>z_d$,  $\pi^z(L_{1/2}\text{ percolates })=1$.
\end{proposition}
 
The value $z_d$ is called the percolation threshold of the Poisson Boolean model with radius $1/2$. By scale invariance, the percolation threshold for any other radius $R$ is simply $z_d/(2R)^d$.  The exact value of $z_d$ is unknown but numerical studies provide for instance the approximation $z_2\simeq 1.4$ in dimension $d=2$.

\begin{theorem} \label{theoreunicite} Let $H$ be an energy function with finite range $R>0$ such that the local energy $h$ is uniformly bounded from below by a constant $C$. Then for any $\beta\ge 0$ and $z<z_d e^{C\beta} /R^d$, there exists an unique Gibbs measure with energy function $H$, activity $z>0$ and inverse temperature $\beta$. 
\end{theorem} 

\begin{proof}

The proof is based on two main ingredients. The first one is the stochastic domination of Gibbs measures, with uniformly bounded from below local energy function $h$, by Poisson processes. This result is given in the following lemma, whose proof can be found in \cite{GK}. The second ingredient is a disagreement percolation result presented in Lemma \ref{disagreement} below.

\begin{lemma}\label{LemmaGK}
 Let $H$ be an energy function such that the local energy $h$ is uniformly bounded from below by a constant $C$.  Then for a any bounded set $\D$ and any outside configuration $\g_{\D^c}$ the Gibbs distribution inside $\D$ given by 
$$P^{z,\beta}(d\g_\D|\g_{\D^c}) = \frac{1}{Z_{\D}^{z,\beta}(\g_{\D^c})} z^{N_\D(\g)} e^{-\beta H_\D(\g)} \pi_\D(d\g_\D)$$
is stochastically dominated by the Poisson point distribution  $\pi_\D^{ze^{-C\beta}}(d\g_\D)$.
\end{lemma}

Thanks to Strassen's Theorem, this stochastic domination can be interpreted via the  following coupling (which could be the definition of the stochastic domination): There exist two point processes $\G$ and $\G'$ on $\Delta$ such that $\G\subset \G'$, $ \G\sim P^{z,\beta}(d\g_\D|\g_{\D^c})$ and $\G'\sim  \pi_\D^{ze^{-C\beta}}(d\g_\D)$.

Now the rest of the proof of Theorem \ref{theoreunicite} consists 
in showing that the Gibbs measure is unique as soon as $\pi^{ze^{-C\beta}}(L_{R/2} \text{percolates})=0$. Roughly speaking, if the dominating process does not percolate, the information coming from the boundary condition does not propagate in the heart of the model and the Gibbs measure is unique. To prove rigorously this phenomenon, we need a disagreement percolation argument introduced first in  \cite{MV94}. For any sets $A,B\in\R^d$, we denote by $A\ominus B$ the set $(A^c\oplus B)^c$.

\begin{lemma}\label{disagreement}
Let $\g_{\D^c}^1$ and $\g_{\D^c}^2$ be two  configurations on $\D^c$. For any $R'>R$, there exist three point processes $\G^1$, $\G^2$ and $\G'$ on $\Delta$ such that $\G^1\subset \G'$, $\G^2\subset \G'$, $ \G^1\sim P^{z,\beta}(d\g_\D|\g^1_{\D^c})$,  $ \G^2\sim P^{z,\beta}(d\g_\D|\g^2_{\D^c})$ and $\G'\sim  \pi_\D^{ze^{-C\beta}}(d\g_\D)$. Moreover, denoting by $L^\D_{R'/2}(\G')$ the connected components of $L_{R'/2}(\G')$ which are inside $\D\ominus B(0,R'/2)$, then $\G^1=\G^2$ on the set  $L^\D_{R'/2}(\G')$.
\end{lemma}

\begin{proof}
Let us note first that, by Lemma \ref{LemmaGK}, there exist three point processes $\G^1$, $\G^2$ and $\G'$ on $\Delta$ such that $\G^1\subset \G'$, $\G^2\subset \G'$, $ \G^1\sim P^{z,\beta}(d\g_\D|\g^1_{\D^c})$,  $ \G^2\sim P^{z,\beta}(d\g_\D|\g^2_{\D^c})$ and $\G'\sim  \pi_\D^{ze^{-C\beta}}(d\g_\D)$. The main difficulty is now to show that we can build
 $\G^1$ and $\G^2$  such that $\G^1=\G^2$ on the set  $L^\D_{R'/2}(\G')$. 
 
 Let us decompose $\D$ via a grid of small cubes where each cube has a diameter smaller than $\epsilon=(R'-R)/2$. We define an arbitrary numeration of these cubes $(C_i)_{1\le i\le m}$ and we construct progressively the processes  $\G^1$, $\G^2$ and $\G'$ on each cube $C_i$. Assume that they are already constructed on $C_I:=\cup_{i\in I} C_i$ with all the expected properties:$\G^1_{C_I}\subset \G'_{C_I}$, $\G^2_{C_I}\subset \G'_{C_I}$, $ \G^1_{C_I}\sim P^{z,\beta}(d\g_{C_I}|\g^1_{\D^c})$,  $ \G^2_{C_I}\sim P^{z,\beta}(d\g_{C_I}|\g^2_{\D^c})$, $\G'_{C_I}\sim  \pi_{C_I}^{ze^{-C\beta}}(d\g_{C_I})$ and $\G^1_{C_I}=\G^2_{C_I}$ on the set  $L^\D_{R'/2}(\G'_{C_I})$.   Let us consider the smaller index $j\in\{1,\ldots m\}\backslash I$ such that either the distances $d(C_j,\g_{\D^c}^1)$ or $d(C_j,\g_{\D^c}^2)$ or $d(C_j,\G'_{C_I})$ is smaller than $R'-\epsilon$. 
 
 \begin{itemize}
 \item If such an index $j$ does not exist, by the finite range property the following  Gibbs distributions coincide  on $\D^I=\D\backslash C_I$;

$$ P^{z,\beta}(d\g_{\D^I}|\g^1_{\D^c}\cup \G^1_{C_I})=  P^{z,\beta}(d\g_{\D^I}|\g^2_{\D^c}\cup \G^2_{C_I}).$$
 
Therefore we define $\G^1$, $\G^2$ and $\G'$ on $\D^I$ by considering $\G_{\D^I}^1$ and $\G_{\D^I}'$ as in Lemma \ref{LemmaGK} and by putting $\G_{\D^I}^2=\G_{\D^I}^1$. We can easily check that all expected properties hold and the full construction of $\G^1$, $\G^2$ and $\G'$ is over.

\item If such an index $j$ does exist, we consider the double coupling construction of $\G^1$, $\G^2$ and $\G'$ on $\D^I$. It means that $\G^1_{\D^I}\subset \G'_{\D^I}$, $\G^2_{\D^I}\subset \G'_{\D^I}$, $ \G^1_{\D^I}\sim  P^{z,\beta}(d\g_{\D^I}|\g^1_{\D^c}\cup \G^1_{C_I})$,  $\G^2_{\D^I}\sim P^{z,\beta}(d\g_{\D^I}|\g^2_{\D^c}\cup \G^2_{C_I})$ and $\G'_{\D^I}\sim  \pi_{\D^I}^{ze^{C\beta}}(d\g_\D)$. Now we keep these processes $\G^1_{\D^I}$, $\G^2_{\D^I}$ and $\G'_{\D^I}$ only on the window $C_j$. The construction of the processes $\G^1$, $\G^2$ and $\G'$ is now over $C_I\cup C_j$ and we can check again that all expected properties hold. We go on to the construction of the processes on a new cube in $(C_i)_{i\in \{1,\ldots n\}\backslash \{I,j\}}$ and so on.
\end{itemize}
\end{proof}

Let us now finish the proof of Theorem \ref{theoreunicite} by considering two infinite volume GPP $\tilde\G^1$ and $\tilde\G^2$ with distribution $P^1$ and $P^2$. We have to show that for any local event $A$ $P^1(A)=P^2(A)$. We denote by $\D_0$ the support of such an event $A$. Let us consider a bounded subset $\D\supset\D_0$ and three new processes $\G_\D^1$, $\G_\D^2$ and $\G_\D'$ on $\Delta$ constructed as in Lemma \ref{disagreement}. Precisely, for any $i=1,2$  
$\G_\D^i\subset \G_\D'$, $\G'\sim  \pi_\D^{ze^{-C\beta}}$, the conditional distribution of $ \G_\D^i$ given $\tilde \G^i_{\D^c}$ is $P^{z,\beta}(|\tilde\G^i_{\D^c})$ and $\G_\D^1=\G_\D^2$ on the set  $L^\D_{R'/2}(\G_\D')$. The parameter $R'>R$ is chosen such that

\begin{equation}\label{choixR}
ze^{-C\beta} R'^d<z_d
\end{equation}
which is possible by assumption on $z$. 

Thanks to the DLR equations \eqref{DLRinfiniteVolume}, for any $i=1,2$ the processes $\G^i_\D$ and $\tilde \G^i_\D$ have the same distributions and therefore $ P^i(A)=P(\G_\D^i\in A)$. Denoting by $\{\Delta\leftrightarrow\D_0\}$ the event that there exists a connected component in $L_{R'/2}(\G')$ which intersects $(\D\ominus B(0,R'/2))^c$ and $\D_0$, we obtain that

\begin{eqnarray}\label{domination}
|P^1(A)-P^2(A)| & =& |P(\G_\D^1\in A)-P(\G_\D^2\in A)|\nonumber \\
&\le & E\bigg(\1_{\{\Delta\leftrightarrow\D_0\}}\Big|\1_{\G_\D^1\in A}-\1_{\G_\D^2\in A}\Big|\bigg)+E\bigg(\1_{\{\Delta\leftrightarrow\D_0\}^c}\Big|\1_{\G_\D^1\in A}-\1_{\G_\D^2\in A}\Big|\bigg)\nonumber \\
& \le & P(\{\Delta\leftrightarrow\D_0\})+E\bigg(\1_{\{\Delta\leftrightarrow\D_0\}^c}\Big|\1_{\G_\D^1\in A}-\1_{\G_\D^1\in A}\Big|\bigg)\nonumber \\
& =&  P(\{\Delta\leftrightarrow\D_0\}).
\end{eqnarray}

By the choice of $R'$ in inequality \eqref{choixR} and Proposition \ref{PropoPerco}, it follows that  $$\pi^{ze^{-C\beta}}\left(L_{R'/2} \text{percolates}\right)=0$$ and we deduce, by a monotonicity argument, the probability $P(\{\Delta\leftrightarrow\D_0\})$ tends to $0$ when $\Delta$ tends to $\R^d$ (see \cite{meeroy} for details on equivalent characterizations of continuum percolation). The left term in \eqref{domination} does not depend on $\Delta$ and therefore it is null. Theorem \ref{theoreunicite} is proved.

\end{proof}

\subsection{A non-uniqueness result}\label{SectionNonUniqueness}

In this section we discuss the non-uniqueness phenomenon of infinite volume Gibbs measures. It is believed to occur for almost all models provided that the activity $z$ or the inverse temperature $\beta$ is large enough. However, in the present continuous setting without spin, it is only proved for few models and several old conjectures are still valid. For instance, for the pairwise Lennard-Jones interaction defined in \eqref{energypairwise}, it is conjectured that for $\beta$ large (but not too large) there exists an unique $z$ such that the Gibbs measures are not unique. It would correspond to a liquid-vapour phase transition. Similarly for $\beta$ very large, it is conjectured that the non-uniqueness occurs as soon as $z$ is larger than a threshold $z_\beta$. It would correspond to a crystallization phenomenon  for which a symmetry breaking may occur. Indeed, it is expected, but not proved at all, that some continuum Gibbs measures would be not invariant under symmetries like  translations, rotations, etc. This conjecture is probably one of the most important and difficult challenges in statistical physics. In all cases, the non-uniqueness appear when the local distribution of  infinite volume Gibbs measures depend on the boundary conditions "at infinity".

In this section we give a complete proof of such non-uniqueness result for the Area energy interaction presented in $\eqref{EnergyArea}$. This result has been first proved in \cite{Widom70} but our proof is inspired by the one given in \cite{CCK}. Roughly speaking, we build two different Gibbs measures which depend, via a percolation phenomenon, on the  boundary conditions "at infinity". In one case, the boundary condition "at infinity" is empty and in the other case the boundary condition is full of particles. We show that the intensity of both infinite volume Gibbs measures are different.

 Let us cite another famous non-uniqueness result   for attractive pair and repulsive four-body potentials \cite{LMP}. As far as we know, this result and the one presented below on the Area interaction, are the only rigorous proofs of non-uniqueness results for continuum particles systems without spin.

 \begin{theorem} \label{TheoremeNonUniqueness}
 For $z=\beta$ large enough, the infinite volume Gibbs measures for the Area energy function $H$ presented in \eqref{EnergyArea}, the activity $z$ and the inverse temperature $\beta$ are not unique.
\end{theorem} 

\begin{proof}

In all the proof we fix $z=\beta$. Let us consider following finite volume Gibbs measures on $\L_n=[-n,n]^d$ with different boundary conditions:

$$ dP_{\L_n}(\g)=\frac{1}{Z_{\L_n}}\1_{\left\{\g_{\L_n\backslash \L_n^\ominus}=\emptyset\right\}} z^{N_{\L_n}(\g)}e^{-z\text{Area}\big(\L_n\cap L_R(\g)\big)}d\pi_{\L_n}(\g),$$

and

$$ dQ_{\L_n}(\g)=\frac{1}{ Z'_{\L_n}}z^{N_{\L_n}(\g)} e^{-z\text{Area}\big(\L_n^\ominus \cap L_R(\g)\big)}d\pi_{\L_n}(\g),$$

where $\L_n^\ominus=\L_n \ominus B(0,R/2)$. Recall that $R$ is the radius of balls in $L_R(\g)=\cup_{x\in\g} B(x,R)$ and that the range of the interaction is $2R$. As in Section \ref{SectionAccu} we consider the associated empirical fields $\bar P_{\L_n}$ and $\bar Q_{\L_n}$ defined by
$$\int f(\g)d\bar P_{\L_n}(\g)= \frac{1}{\l^d(\L_n)} \int_ {\L_n} f(\tau_u(\g)) dP_{\L_n}(\g)du$$
and
$$\int f(\g)d\bar Q_{\L_n}(\g)= \frac{1}{\l^d(\L_n)} \int_ {\L_n} f(\tau_u(\g)) dQ_{\L_n}(\g)du,$$
where $f$ is any measurable bounded test function.
Following the proof of Proposition \ref{existenceAccu} we get the existence of an accumulation point $\bar P$ (respectively $\bar Q$) for $(\bar P_{\L_n})$ (respectively $(\bar Q_{\L_n})$). As in Theorem \ref{THexistence}, we show that $\bar P$ and $\bar Q$ satisfy the DLR equations and therefore they are both infinite volume Gibbs measures for the Area energy function, the activity $z$ and the inverse temperature $\beta=z$. Now it remains to prove that $\bar P$ and $\bar Q$ are different when $z$ is large enough. Note that the difference between $\bar P$ and $\bar Q$ comes only from their boundary conditions "at infinity" (i.e. the boundary conditions of $P_{\L_n}$  and $ Q_{\L_n}$ when $n$ goes to infinity).

Let us start with a representation of  $P_{\L_n}$ and $Q_{\L_n}$ via the two type Widom-Rowlinson model on $\L_n$. Consider the following event of allowed configurations on $\C_{\L_n}^2$ 
\begin{equation}\label{allowed}
\A=\left\{ (\g^1, \g^2)\in\ \C_{\L_n}^2, \text{ s.t. }
 \begin{array}{l}
a)\; L_{R/2}(\g^1)\cap L_{R/2}(\g^2)=\emptyset\\
b)\; L_{R/2}(\g^1)\cap \L_n^c=\emptyset 
\end{array}\right\}
\end{equation}
which assumes first that the balls with radii $R/2$ centred at   $\g^1$ and $\g^2$ do not overlap and secondly that the balls centred at $\g^1$ are completely inside $\L_n$. 

 The two type Widom-Rowlinson model on $\L_n$ with boundary condition b) is the probability measure $\tilde P_{\L_n}$ on $\C_{\L_n}^2$ which is absolutely continuous with respect to the product $(\pi^z_{\L_n})^{\otimes 2}$ with density
$$ \frac{1}{\tilde Z_n} \1_\A(\g^1,\g^2) z^{N_{\L_n}(\g^1)} z^{N_{\L_n}(\g^2)}d\pi_{\L_n}(\g^1) d\pi_{\L_n}(\g^2),$$
where $\tilde Z_{\L_n}$ is a normalization factor.

\begin{lemma}\label{representaionWR}  
The  first marginal (respectively the second marginal) distribution of  $\tilde P_{\L_n}$ is  $P_{\L_n}$ (respectively $Q_{\L_n}$).
\end{lemma}
  
\begin{proof}

By definition of $\tilde P_{\L_n}$, its first marginal admits the following unnormalized density with respect to $\pi_{\L_n}(d\g^1)$

\begin{eqnarray*}
f(\g^1)&= &\int  \1_\A(\g^1,\g^2) z^{N_{\L_n}(\g^1)} z^{N_{\L_n}(\g^2)}d\pi_{\L_n}(\g^2)\\
&=&e^{(z-1)\l^d(\L_n)}z^{N_{\L_n}(\g^1)}\int  \1_\A(\g^1,\g^2) d\pi^z_{\L_n}(\g^2)\\
&=& e^{(z-1)\l^d(\L_n)}z^{N_{\L_n}(\g^1)}\1_{\left\{\g^1_{\L_n\backslash \L_n^\ominus}=\emptyset\right\}} e^{-z\text{Area}\big(\L_n\cap L_R(\g^1)\big)}
\end{eqnarray*}
which is proportional to the density of $P_{\L_n}$. A similar computation gives the same result for $Q_{\L_n}$.
 
\end{proof}

 Now let us give a representation of the two type Widom-Rowlinson model via the random cluster model. The random cluster process $R_{\L_n}$ is a point process on $\L_n$ distributed by
 
 $$ \frac{1}{\hat Z_n} z^{N_{\L_n}(\g)} 2^{N^{\L_n}_{cc}(\g)}d\pi_{\L_n}(\g),$$  
 where $N^{\L_n}_{cc}(\g)$ is the number of connected components of $L_{R/2}(\g)$ which are completely included in $\L_n$. Then we build two new point processes $\hat\G_{\L_n}^1$ and $\hat\G_{\L_n}^2$ by splitting randomly and uniformly the connected component of $R_{\L_n}$. Each connected component inside $\L_n$ is given to  $\hat\G_{\L_n}^1$  or  $\hat\G_{\L_n}^2$ with probability an half each. The connected components hitting $\L_n^c$ are given to $\hat\G_{\L_n}^2$. Rigorously this construction is done by the following way. Let us consider $(C_i(\g))_{1\le i\le N^{\L_n}_{cc}(\g)}$ the collection of connected components of $L_{R/2}(\g)$ inside $\L_n$. Let $(\epsilon_i)_{i\ge 1}$ be a sequence of independent Bernoulli random variables with parameter $1/2$. The processes $\hat\G_{\L_n}^1$ and $\hat\G_{\L_n}^2$ are defined by
 
 $$ \hat\G_{\L_n}^1=\bigcup_{1\le i\le N^{\L_n}_{cc}(R_{\L_n}),\; \epsilon_i=1} R_{\L_n}\cap C_i(R_{\L_n}) \qquad \text{ and } \qquad  \hat\G_{\L_n}^2=R_{\L_n} \backslash  \hat\G_{\L_n}^1.  $$

\begin{lemma}\label{representationCRCM}
 The distribution of $(\hat\G_{\L_n}^1,\hat\G_{\L_n}^2)$ is the two-type Widom-Rowlinson model with boundary condition b). In particular, $\hat\G_{\L_n}^1\sim P_{\L_n}$ and $\hat\G_{\L_n}^2\sim Q_{\L_n}$.
\end{lemma}

\begin{proof}
For any bounded measurable test function $f$ we have

\begin{eqnarray*}
& & E(f(\hat\G_{\L_n}^1,\hat\G_{\L_n}^2))\\
 & = & E\left[f\left(\bigcup_{1\le i\le N^{\L_n}_{cc}(R_{\L_n}),\; \epsilon_i=1} R_{\L_n}\cap C_i(R_{\L_n}), R_{\L_n}\cap\Big( \bigcup_{1\le i\le N^{\L_n}_{cc}(R_{\L_n}),\; \epsilon_i=1} C_i(R_{\L_n})\Big)^c \right)\right]\\
&= & \frac{1}{\hat Z_n} \int \sum_{(\epsilon_i)\in\{0,1\}^{N^{\L_n}_{cc}(\g)}} \frac{1}{2^{N^{\L_n}_{cc}(\g)}} \\
& & f\left(\bigcup_{1\le i\le N^{\L_n}_{cc}(\g),\; \epsilon_i=1}\g\cap C_i(\g), \g \cap \Big(\bigcup_{1\le i\le N^{\L_n}_{cc}(\g),\; \epsilon_i=1} C_i(\g)\Big)^c \right) z^{N_{\L_n}(\g)}2^{N^{\L_n}_{cc}(\g)} d\pi_{\L_n}(\g)\\
&= & \frac{1}{\hat Z_n} \int \sum_{(\epsilon_x)\in\{0,1\}^{\g}}  (\1_\A f)\left(\bigcup_{x\in\g,\; \epsilon_x=1} \{x\}, \g \backslash \bigcup_{x\in\g,\; \epsilon_x=1} \{x\} \right)  z^{N_{\L_n}(\g)}d\pi_{\L_n}(\g)\\
&= & \frac{1}{\hat Z_n} \int  (\1_\A f)\left(\bigcup_{(x,\epsilon _x)\in\tilde \g,\; \epsilon_x=1} \{x\}, \bigcup_{(x,\epsilon _x)\in\tilde \g,\; \epsilon_x=0} \{x\} \right) (2z)^{N_{\L_n}(\g)} d\tilde\pi_{\L_n}(\tilde\g)
\end{eqnarray*}
where  $\tilde\pi_{\L_n}$ is a marked Poisson point process on $\L_n\times\{0,1\}$. It means that the points are distributed by $\pi_{\L_n}$ and that each point $x$ is marked independently by a Bernoulli variable $\epsilon_x$ with parameter $1/2$. We obtain

\begin{eqnarray*}
E(f(\hat\G_{\L_n}^1,\hat\G_{\L_n}^2)) & = & \frac{e^{|\L_n|}}{\hat Z_n} \int  (\1_\A f)\left(\bigcup_{(x,\epsilon _x)\in\tilde \g,\; \epsilon_x=1} \{x\}, \bigcup_{(x,\epsilon _x)\in\tilde \g,\; \epsilon_x=0} \{x\} \right) z^{N_{\L_n}(\g)} d\tilde\pi_{\L_n}^{2}(\tilde\g)\\
& =& \frac{e^{|\L_n|}}{\hat Z_n} \int  \int (\1_\A f)\left(\g^1,\g^2\right) z^{N_{\L_n}(\g^1)}z^{N_{\L_n}(\g^2)} d\pi_{\L_n}(\g^1) d\pi_{\L_n}(\g^2),
\end{eqnarray*}
which proves the Lemma.
\end{proof}

Note that the random cluster process $R_{\L_n}$ is a finite volume GPP with energy function $\hat H=-N_{cc}^{\L_n}$, activity $z$ and inverse temperature $\log(2)$. Its local energy $\hat h$ is defined by 

$$ \hat h(x,\g)= N^{\L_n}_{cc}(\g)-N^{\L_n}_{cc}(\g\cup\{x\}).$$

Thanks to a geometrical argument, it is not difficult to note that $\hat h$ is uniformly bounded from above by a constant $c_d$ (depending only on the dimension $d$). For instance, in the case $d=2$, a ball with radius $R/2$ can overlap at most 5 disjoints balls with radius $R/2$ and therefore $c_2=5-1=4$ is suitable.
  
By Lemma \ref{LemmaGK}, we deduce that the distribution of  $R_{\L_n}$ dominates the Poisson point distribution $\pi_{\L_n}^{2ze^{-c_d}}$. So we choose
$$  z> \frac{z_de^{c_d}}{2R^d} $$

which implies that the Boolean model with intensity $2ze^{-c_d}$ and radii $R/2$ percolates with probability one (see Proposition \ref{PropoPerco}). For any $\g\in\C$, we denote by $C_{\infty}(\g)$ the unbounded connected components in $L_{R/2}(\g)$ (if it exists) and we define by $\alpha$ the intensity of points in $C_{\infty}(\g)$ under the distribution  $\pi^{2ze^{-c_d}}$; 

\begin{equation}\label{alpahpsitif}
\alpha:= \int N_{[0,1]^d}\Big(\g\cap C_\infty(\g)  \Big)d\pi^{2ze^{-c_d}}(\g)>0.
\end{equation}

We are now in position to finish the proof of Theorem \ref{TheoremeNonUniqueness} by proving that the difference in intensities between $\bar Q$ and $\bar P$ is larger than $\alpha$.

The local convergence topology $\tL$ ensures that, for any local bounded function $f$, the evaluation $P\mapsto \int fdP$ is continuous. Actually, the continuity of such evaluation holds for the larger class of functions $f$ satisfying: i) $f$ is local on some bounded set $\D$ ii) there exists $A>0$ such that $|f(\g)|\le A(1+\#(\g))$. In particular, the application $P\mapsto i(P):=\int N_{[0,1]^d}(\g)P(d\g)$ is continuous (see \cite{GZ} for details). We deduce that
  
\begin{eqnarray*}
i(\bar Q)-i(\bar P) & = & \int N_{[0,1]^d}(\g)d\bar Q(\g)-  \int N_{[0,1]^d}(\g)d\bar P(\g) \\
& = & \lim_{n\to\infty} \left( \int N_{[0,1]^d}(\g)d\bar Q_{\L_n}(\g)-  \int N_{[0,1]^d}(\g)d\bar P_{\L_n}(\g)\right)\\
& = & \lim_{n\to\infty} \frac{1}{\l^d(\L_n)} \int_{\L_n} \left( \int N_{[0,1]^d}(\tau_u\g)dQ_{\L_n}(\g)\right.\\
& & \qquad \qquad\qquad-  \left.\int N_{[0,1]^d}(\tau_u\g) dP_{\L_n}(\g)\right)du.\\
 \end{eqnarray*}
 
 By the representation of $P_{\L_n}$ and $Q_{\l_n}$ given in Lemma \ref{representationCRCM}, we find

\begin{eqnarray*}
i(\bar Q)-i(\bar P) & = & \lim_{n\to\infty} \frac{1}{\l^d(\L_n)} \int_{\L_n} E\left(  N_{[0,1]^d}(\tau_u\hat\G_{\L_n}^2)-N_{[0,1]^d}(\tau_u\hat\G_{\L_n}^1) \right)du\\
& = & \lim_{n\to\infty} \frac{1}{\l^d(\L_n)} \int_{\L_n} E\left(  N_{\tau_u[0,1]^d}(R_{\L_n}\cap C_{b}(R_{\L_n})) \right)du,\\
\end{eqnarray*}  
where $C_{b}(\g)$ are the connected components of $L_{R/2}(\g)$ hitting $\L_n^c$. Since the distribution of $R_{\L_n}$ dominates $\pi_{\L_n}^{2ze^{-c_d}}$,

\begin{eqnarray*}
i(\bar Q)-i(\bar P) & \ge  & \lim_{n\to\infty} \frac{1}{\l^d(\L_n)} \int_{[-n,n-1]^d} \int N_{\tau_u[0,1]^d}\Big(\g\cap C(\g)_\infty \Big)d\pi_{\L_n}^{2ze^{-c_d}}(\g)du,\\
 & \ge  & \lim_{n\to\infty} \frac{1}{\l^d(\L_n)} \int_{[-n,n-1]^d}\alpha du=\alpha>0.
\end{eqnarray*}  
  
 The theorem is proved. 
\end{proof}

\section{Estimation of parameters.}\label{SectionEstimation}

In this section we investigate the parametric estimation of the activity $z^*$ and the inverse temperature $\beta^*$ of an infinite volume Gibbs point process $\G$. As usual the star  specifies that the parameters  $z^*,\beta^*$ are unknown whereas the variable $z$ and $\beta$ are used for the optimization procedures. Here the dataset is the observation of $\G$ trough the bounded window $\L_n=[-n,n]^d$ (i.e. the process $\G_{\L_n}$). The asymptotic means that the window $\L_n$ increases to the whole space $\Rd$ (i.e. $n$ goes to infinity) without changing the realization of $\G$. 

For sake of simplicity, we decide to treat only the case of two parameters $(z,\beta)$ but it would be possible to consider energy functions depending on an extra parameter $\theta\in\R^p$.
The case where $H$ depends linearly on $\theta$ can be treated exactly as $z$ and $\beta$. For the non linear case the setting is much more complicated and each procedure has to be adapted. References are given in each section. 

In all the section, we assume that the energy function $H$ is stationary and has a finite range $R>0$. The existence of $\G$ is therefore guaranteed by Theorem \ref{THexistence}.
The procedures presented below are not affected by the uniqueness or non-uniqueness of the distribution of such GPP.
 
In Section \ref{SectionMLE}, we start by presenting the  natural maximum likelihood estimator. Afterwards, in Section \ref{SectionTF}, we introduce the general Takacs-Fiksel estimator which is a mean-square procedure based on the GNZ equations. The standard maximum pseudo-likelihood estimator is a particular case of such estimator and is presented in Section  \ref{SectionPMLE}. An application to an unobservable issue is treated in Section \ref{SectionEstimatioZArea}. The last Section \ref{SectionVariationalEstimator} is devoted to a new estimator based on a variational  GNZ equation.

\subsection{Maximum likelihood estimator}\label{SectionMLE}

The natural method to estimate the parameters is the likelihood inference. However a practical issue is that the likelihood depends on the intractable partition function. In the case of sparse data, approximations were first proposed in \cite{Ogata1984}, before simulation-based methods have been developed \cite{geyer-moller1994}. Here, we treat only the theoretical aspects of the MLE and these practical issues are not investigated.

\begin{definition}
The maximum likelihood estimator of $(z^*,\beta^*)$ is given for any $n\ge 1$ by

\begin{equation}\label{EstimatorML}
(\hat z_n,\hat \beta_n) = \text{argmax}_
{z>0,\beta\ge 0} \frac{1}{Z_{\L_n}^{z,\beta}} z^{N_{\L_n}(\G)} e^{-\beta H(\G_{\L_n})}.
\end{equation}
\end{definition}

Note that the argmax is not necessarily unique and that the boundary effects are not considered in this version of MLE. Other choices could be considered.

In this section we show the consistency of such estimators. The next natural question concerns the asymptotic distribution of the MLE but this problem is more arduous and is still partially unsolved today. Indeed, Mase \cite{Mase92} and  Jensen \cite{Jensen93}  proved that the MLE is asymptotically normal when the parameters $z$ and $\beta$ are small enough. Without these conditions, phase transition may occur  and some long-range dependence phenomenon can appear. The MLE might then exhibit a non standard asymptotic behavior, in the sense that the rate of convergence might differ from the standard square root of the size of the window and the limiting law might be non-gaussian. 

The next theorem is based on a preprint by Mase \cite{MasePreprint}. See alse \cite{DLAOS} for general results on consistency.

\begin{theorem} \label{TheoremconsistencyMLE}
 We assume that the energy function $H$ is stationary, finite range and not almost surely constant (i.e. there exists a subset $\L\subset \R^d$ such that $H(\g_\L)$ is not $\pi_\L(d\g_\L)$ almost surely constant). We assume also that the mean energy exists for any stationary probability measure $P$ (i.e. the limit \eqref{meanenergy} exists) and that the boundary effects assumption \eqref{boundary} holds. Moreover we assume that for any ergodic Gibbs measure $P$, the following limit holds for $P$-almost every $\g$

 \begin{equation}\label{assumptionergodic}
 \lim_{n\mapsto \infty} \frac{1}{\lambda^d(\L_n)}H(\g_{\L_n})=H(P).
 \end{equation}
  
Then, almost surely the parameters $(\hat z_n,\hat \beta_n)$ converge to $(z^*,\beta^*)$ when $n$ goes to infinity.
\end{theorem}

\begin{proof}

Let us assume that the Gibbs distribution $P$ of $\G$ is ergodic. Ortherwise $P$ can be represented as a mixture of ergodic stationary Gibbs measures (see \cite{Preston}, Theorem 2.2 and 4.1). Therefore the proof of the consistency of the MLE reduces to the case  when $P$ is ergodic, which is assumed henceforth.

Let us consider the log-likelihood contrast function

$$K_n(\theta,\beta)=-\log(Z_{\L_n}^{e^{-\theta},\beta}) -\theta N_{\L_n}(\G) -\beta H(\G_{\L_n})$$
related to the parametrization $\theta=-\log(z)$. It is clear that  $(\hat z_n,\hat\beta_n)=(e^{-\tilde \theta_n},\tilde \beta_n)$ where $(\tilde \theta_n,\tilde\beta_n)$ is the argmax of $(\theta,\beta)\mapsto K_n(\theta,\beta)$. So it is sufficient to show that  $(\tilde \theta_n,\tilde\beta_n)$ converges almost surely to $(-\log(z^*),\beta^*)$. The limit \eqref{pressure}, the ergodic Theorem and the assumption \eqref{assumptionergodic} imply the existence of the following limit contrast function

$$K(\theta,\beta):=-p^{e^{-\theta},\beta} -\theta E_P(N_{[0,1]^d}(\G)) -\beta H(P)=\lim_{n\to \infty} \frac{K_n(\theta,\beta)}{\lambda^d(\L_n)}. $$

The variational principle (Theorem \ref{TheoremVP}) ensures that $(\theta,\beta)\mapsto K(\theta,\beta)$ is lower than $I_1(P)$ with equality if and only if $P$ is a Gibbs measure with energy function $H$, activity $z$ and inverse temperature $\beta$. Since $H$ is not almost surely constant, it is easy to see that two Gibbs measures with different parameters $z,\beta$ are different (this fact can be viewed used the DLR equations in a very large box $\L$). Therefore $K(\theta,\beta)$ is maximal, equal to $I_1(P)$, if and only if $(\theta, \beta)=(\theta^*, \beta^*)$.

Therefore it remains to prove that the maximizers of $(\theta,\beta)\mapsto K_n(\theta,\beta)$ converge to the unique maximizer of $(\theta,\beta)\mapsto K(\theta,\beta)$.
First note that the functions $K_n$ are concave. Indeed, the Hessian of $K_n$ is negative since 

$$ \frac{\partial^2 K_n(\theta,\beta)}{\partial^2 \theta} =-\text{Var}_{P_{\L_n}^{e^{-\theta},\beta}}(N_{\L_n}),\quad \frac{\partial^2 K_n(\theta,\beta)}{\partial^2 \beta} =-\text{Var}_{P_{\L_n}^{e^{-\theta},\beta}}(H)$$
and
$$ \frac{\partial^2 K_n(\theta,\beta)}{\partial \theta\partial \beta} =-\text{Cov}_{P_{\L_n}^{e^{-\theta},\beta}}(N_{\L_n},H).$$

The convergence result for the argmax follows since the function $(\theta,\beta)\mapsto K(\theta,\beta)$ is necessarily strictly concave at $(\theta^*,\beta^*)$ because $K(\theta,\beta)$ is maximal uniquely at  $(\theta^*,\beta^*)$.

\end{proof}

Let us finish this section with a discussion on the extra assumption \eqref{assumptionergodic} which claims that the empirical mean energy converges to the expected value energy. This assumption is in general proved via the ergodic theorem or a law of large numbers. In the case of the Area energy function $H$ defined in \eqref{EnergyArea}, it is a direct consequence of a decomposition as in \eqref{decompositionArea} and the ergodic Theorem. In the case of pairwise interaction, the verification follows essentially the proof of Proposition \ref{PropositionVPpairwise}.

\subsection{Takacs-Fiksel estimator} \label{SectionTF}

In this section we present an estimator introduced in the eighties by Takacs and Fiksel \cite{Fiksel, Takacs}. It is based on the GNZ equations presented in Section \ref{SectionGNZinfinite volume}. Let us start by explaining briefly the procedure. Let $f$ be a test function from $\R^d\times\C$ to $\R$. We define the following quantity for any $z>0$, $\beta>0$ and $\g\in\C$ 

\begin{equation}\label{QuantityTF}
C_{\L_n}^{z,\beta}(f,\g)=\sum_{x\in\g_{\L_n}} f(x,\g\backslash \{x\})-z\int_{\L_n} e^{-\beta h(x,\g)}f(x,\g) dx.
\end{equation}

By the GNZ equation \eqref{GNZequtaions} we obtain 

$$E\Big(C^{z^*,\beta^*}_{\L_n}(f,\G)\Big)=0$$
where $\G$ is a GPP with parameter $z^*$ and $\beta^*$. Thanks to the ergodic Theorem it follows that for $n$ large enough

$$
 \frac{C^{z^*,\beta^*}_{\L_n}(f,\G)}{\lambda^d(\L_n)} \approx 0.$$

Then the Takacs-Fiksel estimator is defined as a mean-square method based on functions $C^{z^*,\beta^*}_{\L_n}(f_k,\G)$ for a collection of test functions $(f_k)_{1\le k \le K}$.

\begin{definition}\label{TakacsFikselEstimator}
Let $K\ge 2$ be an integer and $(f_k)_{1\le k \le K}$ a family of $K$ functions from  $\R^d\times\C$ to $\R$. The Takacs-Fiksel estimator $(\hat z_n,\hat \beta_n)$ of $(z^*,\beta^*)$ is defined by

$$ (\hat z_n,\hat \beta_n)=\text{argmin}_{(z,\beta)\in\DD} \sum_{k=1}^K \Big(C^{z,\beta}_{\L_n}(f_k,\G)\Big)^2,$$
\end{definition}
where $\DD\subset(0,+\infty)\times[0,+\infty)$ is a bounded domain containing $(z^*,\beta^*)$.

In opposition to the MLE procedure, the contrast function does not depend on the partition function. This estimator is explicit  except for the computation of integrals and the optimization procedure. In \cite{CDDL} the Takacs-Fiksel procedure is presented in a more general setting including the case where the functions $f_k$ depend on parameters $z$ and $\beta$. This generalization may lead to a simpler procedure in choosing $f_k$ such that the integral term in \eqref{QuantityTF} is explicitly computable.

In the rest of the section, we prove the consistency of the estimator. General results on consistency and asymptotic normality are developed in \cite{CDDL}.

\begin{theorem}[Consistency] \label{TheoremConsTKest}
We make the following integrability assumption: for any $1\le k \le K$ 
\begin{equation}\label{AssumptintTFC}
E\left( |f_k(0,\G)|(1+|h(0,\G)|) \sup_{(z,\beta)\in\DD}e^{-\beta h(0,\G)}\right)<+\infty.
\end{equation}

We assume also the following identifiability condition: the equality  
\begin{equation}\label{assumptionIdenti}
\sum_{k=1}^K E\Big(f_k(0,\G)\big(ze^{-\beta h(0,\G)}-z^*e^{-\beta^* h(0,\G)}\big)\Big)^2=0
\end{equation}
holds if and only $(z,\beta)=(z^*,\beta^*)$.
Then the Takacs-Fiksel estimator $(\hat z_n,\hat \beta_n)$ presented in Definition \ref{TakacsFikselEstimator} converges almost surely to $(z^*,\beta^*)$.

\end{theorem}

\begin{proof}

As in the proof of Theorem \ref{TheoremconsistencyMLE}, without loss of generality, we assume that the Gibbs distribution of $\G$ is ergodic. Therefore, thanks to the ergodic Theorem, almost surely for any $1\le k \le K$

\begin{equation}\label{formule1}
\lim_{n\mapsto \infty} \frac{C^{z,\beta}_{\L_n}(f_k,\G)}{\lambda^d(\L_n)} = E\left[\sum_{x\in\G_{[0,1]^d}} f_k(x,\G\backslash x)\right]-zE\left[\int_{[0,1]^d} e^{-\beta h(x,\G)}f_k(x,\G)dx\right].
\end{equation}
By the GNZ equation \eqref{GNZequtaions}

\begin{equation}\label{formule2}
E\left[\sum_{x\in\G_{[0,1]^d}} f_k(x,\G\backslash x)\right] = z^*E\left[\int_{[0,1]^d} e^{-\beta^* h(x,\G)}f_k(x,\G)dx\right].
\end{equation}

Using the stationarity and compiling \eqref{formule1} and \eqref{formule2}, we obtain that the contrast function 
$$ K_n(z,\beta)= \sum_{k=1}^K \left(\frac{C^{z,\beta}_{\L_n}(f_k,\G)}{\lambda^d(\L_n)}\right)^2$$
admits almost surely the limit 
$$\lim_{n\mapsto \infty} K_n(z,\beta) =K(z,\beta):= \sum_{k=1}^K E\Big(f_k(0,\G)\big(ze^{-\beta h(0,\G)}-z^*e^{-\beta^* h(0,\G)}\big)\Big)^2,$$
which is null if and only if  $(z,\beta)=(z^*,\beta^*)$. Therefore it remains to prove that the minimizers of the contrast function converge to the minimizer of the limit contrast function. In the previous section we solved a similar issue for the MLE procedure using the convexity of contrast functions. This argument does not work here and we need more sophisticated tools. 

We define by $W_n(.)$ the modulus of continuity of the contrast function $K_n$; let $\eta$ be a positive real

$$ W_n(\eta)=\sup \Big\{|K_n(z,\beta)-K_n(z',\beta')|, \text{with } (z,\beta),(z',\beta')\in \DD,\; \Vert(z-z',\beta-\beta')\Vert \le \eta\Big\}.$$

\begin{lemma}[Theorem 3.4.3 \cite{Guyon}]
Assuming that there exists a sequence $(\epsilon_l)_{l\ge 1}$, which goes to zero when $l$ goes to infinity, such that for any $l\ge 1$

\begin{equation}\label{conditionGuyon}
P \left( \limsup_{n\mapsto +\infty} \left\{W_n\left(\frac{1}{l}\right)\ge \epsilon_l\right\} \right)=0
\end{equation}
then almost surely the minimizers of $(z,\beta)\mapsto K_n(z,\beta)$ converges to the minimizer of $(z,\beta)\mapsto K(z,\beta)$.

\end{lemma}

Let us show that the assertion \eqref{conditionGuyon} holds. Thanks to equalities \eqref{formule1}, \eqref{formule2} and  assumption \eqref{AssumptintTFC}, there exists a constant $C_1$ such that for $n$ large enough, any $1\le k \le K$ and any $(z,\beta)\in\DD$

\begin{equation}\label{dominationC}
  \frac{|C^{z,\beta}_{\L_n}(f_k,\G)|}{\lambda^d(\L_n)} \le C_1.
\end{equation}
We deduce that for $n$ large enough

\begin{eqnarray*}
|K_n(z,\beta)-K_n(z',\beta')|& \le & \frac{C_1}{\lambda^d(\L_n)} \sum_{k=1}^K \int_{\L_n} |f_k(x,\G)|\left|ze^{-\beta h(x,\G)}-z'e^{-\beta' h(x,\G)}\right| dx\\
& \le & \frac{C_1|\beta-\beta'|}{\lambda^d(\L_n)} \max_{1\le k\le K}\int_{\L_n} |f_k(x,\G)h(x,\G)|\sup_{(z,\beta'')\in\DD} ze^{-\beta'' h(x,\G)}dx\\
& & +\frac{C_1|z-z'|}{\lambda^d(\L_n)} \max_{1\le k\le K}\int_{\L_n} |f_k(x,\G)|\sup_{(z,\beta'')\in\DD} e^{-\beta'' h(x,\G)}dx.\\
\end{eqnarray*}

By the ergodic Theorem, the following convergences hold almost surely

\begin{eqnarray*}
& & \lim_{n\mapsto +\infty } \frac{1}{\lambda^d(\L_n)} \int_{\L_n} |f_k(x,\G)h(x,\G)|\sup_{(z,\beta'')\in\DD} ze^{-\beta'' h(x,\G)}dx\\
&=& E\left(|f_k(0,\G)h(0,\G)|\sup_{(z,\beta'')\in\DD} ze^{-\beta'' h(0,\G)}\right)<+\infty, 
\end{eqnarray*}
and 

\begin{eqnarray*}
& & \lim_{n\mapsto +\infty } \frac{1}{\lambda^d(\L_n)} \int_{\L_n} |f_k(x,\G)|\sup_{(z,\beta'')\in\DD} e^{-\beta'' h(x,\G)}dx\\
&=& E\left(|f_k(0,\G)|\sup_{(z,\beta'')\in\DD} e^{-\beta'' h(0,\G)}\right)<+\infty.
\end{eqnarray*}
This implies the existence of a constant $C_2>0$ such that for $n$ large enough, any $1\le k \le K$ and any $(z,\beta)\in\DD$

$$ |K_n(z,\beta)-K_n(z',\beta')| < C_2\Vert (z-z',\beta-\beta')\Vert.$$

The assumption \eqref{conditionGuyon} occurs with the sequence $\epsilon_l=C_2/l$ and Theorem \ref{TheoremConsTKest} is proved.

\end{proof}

\begin{remark}[On the integrability assumption]
The integrability assumption \eqref{AssumptintTFC} is sometimes difficult to check, especially when the local energy $h(0,\g)$ is not bounded from below. For instance in the setting of pairwise energy function $H$ defined in $\eqref{energypairwise}$ with a pair potential $\varphi$ having negative values,  Ruelle estimates \eqref{RuelleEstimatesinfinitevolume} are very useful. Indeed, by stability of the energy function, the potential $\varphi$ is necessary bounded from below by $2A$ and therefore 
$$ E\left(e^{-\beta h(0,\G)}\right)<E\left(e^{-2A\beta N_{B(0,R)}(\G))}\right)<+\infty,$$ where $R$ is the range of the interaction. 
\end{remark}

\begin{remark}[On the identifiability assumption] In the identifiability assumption \eqref{assumptionIdenti}, the sum is null if and only if each term is null. Assuming that the functions are regular enough, each term is null as soon as  $(z,\beta)$ belongs to a 1-dimensional manifold embedded in $\R^2$ containing $(z^*,\beta^*)$. Therefore, assumption \eqref{assumptionIdenti} claims that $(z^*,\beta^*)$ is the unique element of these $K$ manifolds. If $K\le 2$, there is no special geometric argument to ensure that $K$  1-dimensional manifolds in $\R^2$ have an unique intersection point. For this reason, it is recommended to choose $K\ge 3$. See Section 5 in \cite{CDDL} for more details and complements on this identifiability assumption.
\end{remark}

\subsection{Maximum pseudo-likelihood estimator}\label{SectionPMLE}

In this section we present the maximum pseudo-likelihood estimator, which is a particular case of the Takacs-Fiksel estimator. This procedure has been first introduced by Besag in \cite{Besag74}  and popularized by Jensen and Moller in \cite{JensenMoller} and Baddeley and Turner in \cite{BaddeleyTurner}.

\begin{definition} \label{MPLEestimator} The  maximum pseudo-likelihood estimator $(\hat z_n,\hat \beta_n)$ is defined as a Takacs-Fiksel estimator (see Definition \ref{TakacsFikselEstimator}) with $K=2$, $f_1(x,\g)=1$ and $f_2(x,\g)=h(x,\g)$.
\end{definition}

This particular choice of functions $f_1$, $f_2$ simplifies the identifiability assumption \eqref{assumptionIdenti}. The following theorem is an adaptation of Theorem \ref{TheoremConsTKest} in the present setting of MPLE. The asymptotic normality is investigated first in \cite{JensenKunsch} (see also \cite{BCD2008} for more general results).

\begin{theorem}[Consistency] \label{TheoremConsMPLE}
Assuming
\begin{equation}\label{AssumptintMPLE}
E\left( (1+h(0,\G)^2) \sup_{(z,\beta)\in\DD}e^{-\beta h(0,\G)}\right)<+\infty
\end{equation}
and 
\begin{equation}\label{assumptionIdentiMPLE}
P\Big(h(0,\G)=h(0,\emptyset)\Big)<1,
\end{equation}
then the maximum pseudo-likelihood estimator $(\hat z_n,\hat \beta_n)$ converges almost surely to $(z^*,\beta^*)$.

\end{theorem}

\begin{proof}
Let us check the assumptions of Theorem \ref{TheoremConsTKest}. Clearly, the integrability assumption \eqref{AssumptintMPLE} ensures  the integrability assumptions \eqref{AssumptintTFC} with $f_1=1$ and $f_2=h$. So it remains to show that assumption \eqref{assumptionIdentiMPLE} implies the identifiability assumption \eqref{assumptionIdenti}. Consider the parametrization $z=e^{-\theta}$ and $\psi$ the function

$$ \psi(\theta,\beta)=E\left(e^{-\theta^*-\beta^*h(0,\G)}(e^U-U-1)\right),$$
with 
$$ U=\beta^*h(0,\G)+\theta^*-\beta h(0,\G)-\theta.$$
The function $\psi$ is convex, non negative and equal to zero if and only if $U$ is almost surely equal to zero. By assumption \eqref{assumptionIdentiMPLE} this fact occurs when $(z,\beta)=(z^*,\beta^*)$. Therefore the gradient $\nabla \psi=0$ if and only $(z,\beta)=(z^*,\beta^*)$. Noting that

$$\frac{\partial \psi(\theta,\beta)}{\partial \theta}=E\left(z^*e^{-\beta^*h(0,\G)}-ze^{-\beta h(0,\G)}\right)$$
and

$$\frac{\partial \psi(\theta,\beta)}{\partial \beta}=E\left(h(0,\G)\left(z^*e^{-\beta^*h(0,\G)}-ze^{-\beta h(0,\G)}\right)\right),$$

the identification assumption \eqref{assumptionIdenti} holds. The theorem is proved.

\end{proof}

\subsection{Solving an unobservable issue}\label{SectionEstimatioZArea}

In this section we give an application of the Takacs-Fiksel procedure in a setting of partially observable dataset. Let us consider a Gibbs point process $\G$ for which we observe only $L_R(\G)$ in place of $\G$. This setting appears when Gibbs point processes  are used for producing random surfaces via  germ-grain structures (see  \cite{MollerHelisova} for instance). Applications for modelling micro-structure in materials or micro-emulsion in statistical physics are developed in \cite{chiu2013}.  

The goal is to furnish an estimator of $z^*$ and $\beta^*$ in spite of this unobservable issue.
Note that the number of points (or balls) is not observable from $L_R(\G)$ and therefore the MLE procedure is not achievable, since the likelihood is not computable.    When $\beta$ is known and fixed to zero, it corresponds to the estimation of the intensity of the Boolean model from its germ-grain structure (see \cite{Molchanov} for instance).

In the following we assume that $\G$ a Gibbs point process for  the Area energy function defined in \eqref{EnergyArea}, the activity $z^*$ and the inverse temperature $\beta^*$. This choice is natural since the energy function depends on the observations $L_R(\G)$. The more general setting of Quermass interaction is presented in \cite{DLS} but for sake of simplicity, we treat only here the simpler case of Area interaction.

We opt for a Takacs-Fiksel estimator but the main problem is that the function

$$
C_{\L_n}^{z,\beta}(f,\g)=\sum_{x\in\g_{\L_n}} f(x,\g\backslash \{x\})-z\int_{\L_n} e^{-\beta h(x,\g)}f(x,\g) dx,$$
which appears in the procedure, is not computable since the positions of points are not observable. The main idea is to choose the function $f$ properly such that the sum is observable although each term of the sum is not. To this end, we define 

$$  f_1(x,\g)=\text{Surface}\Big(\partial B(x,R)\cap L^c_R(\g)\Big)$$
and 
$$ f_2(x,\g)=\1_{\{B(x,R)\cap L_R(\g)=\emptyset\}},$$
where $\partial B(x,R)$ is the boundary of the ball $B(x,R)$ (i.e. the sphere $S(x,R)$) and the "Surface" means the $(d-1)$-dimensional Hausdorff measure in $\Rd$. Clearly the function $f_1$ gives the surface of the portion of the sphere $S(x,R)$ outside the germ-grain structure $L_R(\g)$. The function $f_2$ indicates if the ball $B(x,R)$ hits the germ-grain structure $L_R(\g)$. Therefore we obtain that

$$\sum_{x\in\g_{\L_n}} f_1(x,\g\backslash \{x\})=\text{Surface}\Big(\partial L_R(\g_{\L_n})\Big)$$
and
$$\sum_{x\in\g_{\L_n}} f_2(x,\g\backslash \{x\})=N_{\text{iso}}\Big(L_R(\g_{\L_n})\Big),$$
where $N_{\text{iso}}(L_R(\g_{\L_n}))$ is  the number of isolated balls in the germ-grain structure $L_R(\g_{\L_n})$. Let us note that these quantities are not exactly observable since, in practice, we observe $L_R(\g)\cap \L_n$  rather than $L_R(\g_{\L_n})$. However, if we omit this boundary effect, the values $C_{\L_n}^{z,\beta}(f_1,\G)$ and $C_{\L_n}^{z,\beta}(f_2,\G)$ are observable and the Takacs-Fiksel   procedure  is achievable. The consistency of the estimator is guaranteed by Theorem \ref{TheoremConsTKest}. The integrability assumption \eqref{AssumptintTFC} is trivially satisfied since the functions $f_1$, $f_2$ and $h$ are uniformly bounded. The verification of the identifiability assumption \eqref{assumptionIdenti} is more delicate and we refer to \cite{CDDL}, example 2 for a proof. Numerical estimations on simulated and real datasets can be found in \cite{DLS}.
 
\subsection{A variational estimator}\label{SectionVariationalEstimator}
 
In this last section, we present a new estimator based  on a variational GNZ equation which is a mix between the standard GNZ equation and an integration by parts formula. This equation has been first introduced in \cite{DereudrePHD} for statistical mechanics issues and used recently in \cite{BaddeleyDereudre} for spatial statistic considerations. In the following, we present first this variational equation and afterwards we introduce its associated estimator of $\beta^*$. The estimation of $z^*$ is not considered here. 

\begin{theorem} Let $\G$ be a GPP for the energy function $H$, the activity $z$ and the inverse temperature $\beta$. We assume that, for any $\g\in\C$, the function $x\mapsto h(x,\g)$ is differentiable on $\R^d\backslash \g$. Let $f$ be a function from $\R^d\times \C$ to $\R$  which is  differentiable and with compact support with respect to the first variable. Moreover we assume the integrability of both terms below. Then

\begin{equation}\label{VaritionalGNZ}
E\left(\sum_{x\in\G} \nabla_xf(x,\G\backslash \{x\})\right)= \beta E\left(\sum_{x\in\G} f(x,\G\backslash \{x\})\nabla_xh(x,\G\backslash \{x\})\right).
\end{equation}
\end{theorem}

\begin{proof}

By the standard GNZ equation \eqref{GNZequtaions} applied to the function $\nabla_xf$, we obtain

$$ E\left(\sum_{x\in\G} \nabla_xf(x,\G\backslash \{x\})\right)= z E\left(\int_{\R^d} e^{-\beta h(x,\G)} \nabla_xf(x,\G)dx\right).$$

By a standard integration by part formula with respect to the first variable $x$, we find that

$$ E\left(\sum_{x\in\G} \nabla_xf(x,\G\backslash \{x\})\right)= z\beta E\left(\int_{\R^d} \nabla_xh(x,\G)e^{-\beta h(x,\G)} f(x,\G)dx\right).$$

Using again the GNZ equation we  finally obtain  \eqref{VaritionalGNZ}.
\end{proof}

Note that equation \eqref{VaritionalGNZ} is a vectorial equation. For convenience it is possible to obtain a real equation by summing each coordinate of the vectorial equation. The gradient operator is simply replaced by the divergence operator.

\begin{remark} [on the activity parameter $z$] The parameter $z$ does not appear in the variational GNZ equation \eqref{VaritionalGNZ}. Therefore these equations do not characterize the Gibbs measures as in Proposition \ref{PropGNZreverse}. Actually these variational  GNZ equations characterize the mixing of Gibbs measures with random activity (See \cite{DereudrePHD} for details).

\end{remark}

Let us now explain how to estimate $\beta^*$ from these variational equations. When the observation window $\L_n$ is large enough we identify the expectations of sums in  \eqref{VaritionalGNZ} by the sums. Then the estimator of $\beta^*$  is simply defined by

\begin{equation}\label{varitionalest}
\hat\beta_n= \frac{\sum_{x\in\G_{\L_n}} \text{div}_x f(x,\G\backslash \{x\})}{\sum_{x\in\G_{\L_n}} f(x,\G\backslash \{x\})\text{div}_xh(x,\G\backslash \{x\})}.
\end{equation}
 
Note that this estimator is very simple and quick to compute in comparison to the MLE, MPLE or the general Takacs-Fiksel estimators. Indeed, in \eqref{varitionalest}, there are only elementary operations (no optimization procedure, no integral to compute).

 Let us now finish this section with a consistency result. More general results for consistency, asymptotic normality and practical estimations  are available in \cite{BaddeleyDereudre}.

\begin{theorem} Let $\G$ be a GPP for a stationary and finite range energy function $H$, activity $z^*$ and inverse temperature $\beta^*$. We assume that, for any $\g\in\C$, the function $x\mapsto h(x,\g)$ is differentiable on $\R^d\backslash \g$. Let $f$ be a stationary function from $\R^d\times \C$ to $\R$,  differentiable with respect to the first variable and such that

\begin{equation}\label{assumptionIntVGNZ}
E\left((|f(0,\G|+|\nabla_xf(0,\G)|+|f(0,\G)\nabla_xh(0,\G)|)e^{-\beta^*h(0,\G)}\right)<+\infty
\end{equation}
and
\begin{equation}\label{identityVGNZ}
E\left(f(0,\G)\text{div}_xh(0,\G)e^{-\beta^*h(0,\G)}\right)\neq 0.
\end{equation}

Then the estimator $\hat \beta_n$ converges almost surely to $\beta^*$.
\end{theorem}

\begin{proof}

As usual, without loss of generality, we assume that the Gibbs distribution of $\G$ is ergodic. Then by the ergodic theorem the following limits both hold almost surely

\begin{equation}\label{LIMIT1}
\lim_{n\mapsto +\infty} \frac{1}{\lambda^d(\L_n)} \sum_{x\in\G_{\L_n}} \text{div}_xf(x,\G\backslash \{x\})=E\left(\sum_{x\in\G_{[0,1]^d}} \text{div}_xf(x,\G\backslash \{x\})\right)
\end{equation}
and 

\begin{eqnarray}\label{LIMIT2}
 & &\lim_{n\mapsto +\infty} \frac{1}{\lambda^d(\L_n)} \sum_{x\in\G_{\L_n}} f(x,\G\backslash \{x\})\text{div}_xh(x,\G\backslash \{x\})\nonumber\\
 &=& E\left(\sum_{x\in\G_{[0,1]^d}} f(x,\G\backslash \{x\})\text{div}_xh(x,\G\backslash \{x\})\right).
\end{eqnarray}
Note that both expectations in \eqref{LIMIT1} and \eqref{LIMIT2} are finite since by the GNZ equations, the stationarity and assumption \eqref{assumptionIntVGNZ}

$$ E\left(\sum_{x\in\G_{[0,1]^d}} |\text{div}f(x,\G\backslash \{x\})|\right)=E\left(|\text{div}f(0,\G)|e^{-\beta^*h(0,\G)}\right)<+\infty$$
and 
$$ E\left(\sum_{x\in\G_{[0,1]^d}} |f(x,\G\backslash \{x\})\text{div}h(x,\G\backslash \{x\})|\right)=E\left(|f(0,\G)\text{div}h(0,\G)|e^{-\beta^*h(0,\G)}\right)<+\infty.$$

We deduce that almost surely

$$ \lim_{n\mapsto +\infty} \hat\beta_n= \frac{E\left(\sum_{x\in\G_{[0,1]^d}} \text{div}f(x,\G\backslash \{x\})\right)}{E\left(\sum_{x\in\G_{[0,1]^d}} f(x,\G\backslash \{x\})\text{div}h(x,\G\backslash \{x\})\right)},$$

where the denominator is not null thanks to assumption \eqref{identityVGNZ}. Therefore it remains to prove the following variational GNZ equation

\begin{equation}\label{VGNZS}
 E\left(\sum_{x\in\G_{[0,1]^d}} \nabla_xf(x,\G\backslash \{x\})\right)=\beta^*E\left(\sum_{x\in\G_{[0,1]^d}} f(x,\G\backslash \{x\})\nabla_xh(x,\G\backslash \{x\})\right).
\end{equation} 
Note that this equation is not a direct consequence of the variational GNZ equation \eqref{VaritionalGNZ} since the function $x\mapsto f(x,\g)$ does not have a compact support. We need the following cut-off approximation. Let us consider $(\psi_n)_{n\ge 1}$ any sequence of functions from $\R^d$ to $\R$ such that $\psi_n$ is differentiable, equal to $1$ on $\L_n$, $0$ on $\L_{n+1}^c$ and such that $|\nabla\psi_n|$ and $|\psi_n|$ are uniformly bounded by a constant $C$ (which does not depend on $n$). It is not difficult to build such a sequence of functions. Let us now apply the variational GNZ equation \eqref{VaritionalGNZ} to the function $(x,\g)\mapsto \psi_n(x)f(x,\g)$, we obtain

\begin{eqnarray}\label{VaritionalGNZCut}
& & E\left(\sum_{x\in\G} \psi_n(x)\nabla_xf(x,\G\backslash \{x\})\right)+E\left(\sum_{x\in\G} \nabla_x\psi_n(x)f(x,\G\backslash \{x\})\right)\nonumber\\
& = &  \beta^* E\left(\sum_{x\in\G} \psi_n(x)f(x,\G\backslash \{x\})\nabla_xh(x,\G\backslash \{x\})\right).
\end{eqnarray}

Thanks to the GNZ equation and the stationarity we get

\begin{eqnarray*}
& &\left| E\left(\sum_{x\in\G} \psi_n(x)\nabla_xf(x,\G\backslash \{x\})\right) -\lambda^d(\L_n)E\left(\sum_{x\in\G_{[0,1]^d}} \nabla_xf(x,\G\backslash \{x\})\right)\right| \\
&\le & Cz^*\lambda^d(\L_{n+1}\backslash \L_n) E\left(|\nabla_xf(0,\G)|e^{-\beta^*h(0,\G)}\right),
\end{eqnarray*}   

and

\begin{eqnarray*}
& &\left| E\left(\sum_{x\in\G} \psi_n(x)f(x,\G\backslash \{x\})\nabla_xh(x,\G\backslash \{x\})\right)\right.\\
& & \left. -\lambda^d(\L_n)E\left(\sum_{x\in\G_{[0,1]^d}} f(x,\G\backslash \{x\})\nabla_xh(x,\G\backslash \{x\})\right)\right| \\
&\le & Cz^*\lambda^d(\L_{n+1}\backslash \L_n) E\left(|f(0,\G)\nabla_xh(0,\G)|e^{-\beta^*h(0,\G)}\right),
\end{eqnarray*}   
and finally
$$\left|E\left(\sum_{x\in\G} \nabla_x\psi_n(x)f(x,\G\backslash \{x\})\right)\right|\le Cz^*\lambda^d(\L_{n+1}\backslash \L_n) E\left(|f(0,\G)|e^{-\beta^*h(0,\G)}\right).$$

Therefore, dividing equation \eqref{VaritionalGNZCut} by $\lambda^d(\L_n)$, using the previous approximations and letting $n$ go to infinity, we find exactly the variational equation \eqref{VGNZS}. The theorem is proved.

\end{proof}

{\it Acknowledgement:} The author thanks P. Houdebert, A. Zass and the anonymous referees for the careful reading and the interesting comments. This work was supported in part by the Labex CEMPI (ANR-11-LABX-0007-01), the CNRS GdR 3477 GeoSto and the ANR project PPP (ANR-16-CE40-0016).

\bibliographystyle{plain}
\bibliography{biblioDEREUDRE}

\end{document}